\documentclass[11pt, a4paper, oneside]{memoir}
\usepackage[margin=1in]{geometry}   %margin
\usepackage[utf8]{inputenc}         %æ,ø,å
\usepackage[T1]{fontenc}            %æ,ø,å
\usepackage[english]{babel}         %spell checking
\usepackage{microtype}              %better typography
\usepackage{nicematrix, tikz, tikz-cd}
    \usetikzlibrary{fit}
\usepackage{amsmath,                %diverse matematik
            amssymb, 
            amsthm,
            amsfonts, 
            mathtools,
            upgreek
            }
\usepackage{xcolor,                 %idk
            enumitem,
            tcolorbox}
\usepackage{hyperref, cleveref}     %references
\usepackage{graphicx}               %grafics
\usepackage{xspace}                 %automatic spaces when using macro
\usepackage{xargs}                  %multiple optional arguments

\usepackage{float}

\usepackage{csquotes}

\usepackage{amsrefs}
\theoremstyle{plain}\newtheorem{theorem}[subsection]{Theorem}
\newtheorem{lemma}[subsection]{Lemma}
\newtheorem{proposition}[subsection]{Proposition}
\newtheorem{corollary}[subsection]{Corollary}

\theoremstyle{definition}
\newtheorem{definition}[subsection]{Definition}
\newtheorem{example}[subsection]{Example}
\newtheorem{remark}[subsection]{Remark}

\newcommand{\parenum}{\setenumerate[1]{label=\textbf{(\arabic*)}}}

\newcommand{\twiddle}[1]{\widetilde{#1}}

\newcommand{\E}{\mathbb{E}}

\newcommand{\N}{\mathbb{N}}

\newcommand{\R}{\mathbb{R}}

\newcommand{\ideal}[1]{\left\langle#1\right\rangle}

\usepackage{derivative}

\usepackage{lipsum}
\newcommand{\length}[2][\sigma]{\left.L(#1\right|_{#2})}
\newcommandx{\GTri}[3][1=x,2=y,3=z]{\ensuremath{\Delta(#1,#2,#3)}}
\newcommandx{\CTri}[3][1=x,2=y,3=z]{\ensuremath{\overline{\Delta}(\overline{#1},\overline{#2},\overline{#3})}}

\usepackage{todonotes}

\definecolor{back}{RGB}{255,255,255}
\definecolor{fore}{RGB}{0,0,0}
\definecolor{title}{RGB}{255,0,90}

\definecolor{green}{rgb}{0.0, 0.5, 0.0}
\definecolor{purple}{rgb}{0.5, 0.0, 0.5}
\definecolor{bluegreen}{rgb}{0.0,0.5, 0.5}
\definecolor{orange}{rgb}{1,0.5, 0.1}
\definecolor{redgreen}{rgb}{0.5, 0.5, 0.0}

\author{Søren Poulsen}
\begin{document}
%-----------------title page------------------
\begin{titlingpage}
	\centering
	{\scshape Aarhus Universitet \par}
	\vspace{1cm}
	{\scshape\Large Mathematical project\par}
	\vspace{1.5cm}
	{\huge\bfseries A gentle introduction to CAT($0$) spaces\par}
    \vspace{0.5cm}
    %{\large\bfseries CAT(0)-rum\par}
	\vspace{2cm}
	{\Large\itshape Søren R. B. Poulsen\par}
    {\par}
	\vfill
	supervised by:\par
	Corina-Gabriela \textsc{Ciobotaru}

	\vfill

	{\large 4\textsuperscript{th} of June 2024\par}
\end{titlingpage}

%-----------------introduction--------------
\begin{abstract}
In this project we explore the geometry of general metric spaces, where we do not necessarily have the tools of differential geometry on our side.
Some metric spaces \((X,d)\) allow us to define geodesics, permitting us to compare geodesic triangles in \((X,d)\) to geodesic triangles in a so called model space.
In Chapters 1 and 2 we first discuss how to define the length of curves, and geodesics on \((X,d)\), and then using these to  portray the notion of ``non-positive curvature'' for a metric space.
Chapter 3 concerns itself with special cases of such non-positively curved metric spaces, called CAT($0$) spaces. These satisfy particularly nice properties, such as being uniquely geodesic, contractible, and having a convex metric, among others.

We mainly follow the book \cite{BridHaef} by Martin R. Bridson and André Haefliger, with some differences.
Firstly, we restrict ourselves to using the Euclidean plane \(\E^2\) as our model space, which is all that is necessary to define CAT($0$) spaces.
Secondly, we skip many sections of the mentioned book, as many are not relevant for our specific purposes. 
Finally, we add details to some of the proofs, which can be sparse in details or completely non-existent in the original literature.
In this way we hope to create a more streamlined, self-contained, and accessible introduction to CAT($0$) spaces.
\end{abstract}
\tableofcontents

%----------------length-------------
\chapter{Length spaces}
\section{Length of curves}
Let \((X,d)\) be a metric space and let \(I=[0,a]\subseteq\R\) be a compact interval not the empty set. 
A \emph{curve} is a continuous map \(\sigma:I\to X\) joining the endpoints \(x=\sigma(0)\) and \(y=\sigma(a)\).
In our case, curves are not necessarily smooth, let alone differentiable.
Therefore, the well known notion of arc length \(\ell(\sigma)=\int_a^b|\sigma'(t)|dt\) from undergraduate geometry will not fit our purposes.
A different and more general approach is to ``approximate'' the curve by a polygonal path, the length of which can be determined by applying the metric \(d\) to each line segment.
Intuitively, the more line segments the polygonal path has, the better an approximation of the curve it is.
This is made precise by the following definition:

\begin{definition}
\label{length}
Let \(\sigma:I\to X\) be a curve from the compact interval \(I=[a,b]\subseteq \R\) to a metric space \((X,d)\).
Let \(P\) be a partition \(a=t_0<t_1<\ldots<t_n=b\) of \([a,b]\), and define \(L_P(\sigma)=\sum_{i=1}^n d(\sigma(t_{i-1}),\sigma(t_i))\).
The \emph{length} \(L(\sigma)\) of \(\sigma\) is defined as
$L(\sigma):=\sup L_P(\sigma)$
where the supremum is taken over all possible partitions \(P\) of \(I\).
A curve is said to be \textit{recitifiable} if it has finite length, i.e. \(L(\sigma)<\infty\).
\end{definition}

\medskip
As the name suggests, length provides in a natural way a notion of distance, i.e. a metric on \(X\):
\begin{definition}
Let \((X,d)\) be a metric space.
The \textit{length metric associated with $d$} is the function \(d_i:X\times X\to [0,\infty]\) on \(X\) given by
\[d_i(x,y)=\inf L(\sigma),\]
where the infimum is taken over all rectifiable curves \(\sigma\) joining \(x\) and \(y\). 
If no rectifiable curve between points \(x\) and \(y\) we set \(d_i(x,y)=\infty\).
\end{definition}

\begin{proposition}
The length metric \(d_i\) is a metric on \(X\).
\end{proposition}
\begin{proof}
\textbf{Positive definiteness:} 
Let \(x,y\in X\) be points, and assume that \(d_i(x,y)\neq\infty\).
Then there exists at least one rectifiable curve between \(x\) and \(y\).
Consider first \(x\neq y\).
As \(d\) is a metric it is in particular positively definite, so \(d(x,y)>0\). 
This implies that all summands in the \(L_P(\sigma)\) are positive for any rectifiable curve \(\sigma\) between \(x\) and \(y\), and any partition \(P\) of \(\sigma\), so \(d_i(x,y)>0\) when \(x\neq y\).
In the case \(x=y\), the constant curve mapping all elements in \([a,b]\) to \(x\) is a rectifiable curve from \(x\) to \(y\).
For all partitions we have \(\sigma(t_{i-1})=\sigma(t_i)\), so all summands of \(L_P(\sigma)\) are 0, whereby \(L(\sigma)=0\).
This in turn means that \(0\leq d_i(x,y)\leq L(\sigma)= 0\) so \(d_i(x,y)\) is zero if and only if \(x=y\).

\textbf{Symmetry:}
For all rectifiable curves \(\sigma:[a.b]\to X\) from \(x\) to \(y\) there is a rectifiable curve \(-\sigma:[a,b]\to X\) from \(y\) to \(x\) given by \(-\sigma(t)=\sigma(a+b-t)\).
This can be considered the same curve in reverse.
This curve \(-\sigma\) obviously has the same length as \(\sigma\), so for every curve from \(x\) to \(y\) there is a curve from \(y\) to \(x\) with the same length, and vice versa.
Therefore the infimum of the length of all curves from \(y\) to \(x\) must be the infimum of the length of all curves from \(x\) to \(y\), i.e. \(d_i(x,y)=d_i(y,x)\).
This shows that \(d_i\) is symmetric.

\textbf{Triangle inequality:}
Let \(x,y\) and \(z\) be points in \(X\), let \(\sigma_{a,b}:[a,b]\to X\) be a curve from \(x\) to \(y\), and let \(\sigma_{b,c}:[b,c]\to X\) be a curve from \(y\) to \(z\).
There is then a curve \(\sigma_{a,c}:[a,c]\to X\) which is simply the concatenation of the two curves, i.e.
\[\sigma_{x,z}(t)=\begin{cases}
\sigma_{x,y}(t)     &\text{if }t\in [a,b]\\
\sigma_{y,z}(t)     &\text{if }t\in [b,c].
\end{cases}\]
We clearly have \(L(\sigma_{x,z})\leq L(\sigma_{x,y})+L(\sigma_{y,z})\), whereby \(d_i(x,z)\leq d_i(x,y)+d_i(y,z)\), proving the triangle inequality.
\end{proof}

\begin{definition}\label{def::length metric}
    If \(d=d_i\), the metric space \((X,d)=(X,d_i)\) is called a \textit{length space}. 
\end{definition}

\example{
As curves usually are visualized in \(\R^2\) or \(\R^3\), which are length spaces when equipped with the usual Euclidean metric \(d\), it might be more useful to consider a non-example which nontheless can be visualized.
The unit circle \(S^1\) with the metric \(d_E(x,y)=|x-y|\) induced by the Euclidean metric of \(\R^2\) is \emph{not} a length space for the following reason:
Consider two points \(x,y\in S^1\). 
The length metric \(d_i(x,y)\) is the arc length along the circle \(d_i(x,y)=\arccos(\ideal{x,y})\).
In particular, \(d_i(x,y)\neq d_E(x,y)\), so \((S^1,d_E)\) is \emph{not} a length space.}\label{ex_1}

%---------------geodesics-----------------
\section{Geodesic spaces}
Let \((X,d)\) be a metric space and let \(I=[a,b]\subseteq\R\) be a compact interval not the empty set. 
A curve \(\sigma:I\to X\) is said to have \textit{constant speed \(\lambda\)} if \(L(\sigma|_{[t,t']})=\lambda|t-t'|\) for all \(t\leq t'\in I\).
Intuitively, it takes the same ``time'' (difference in input) to ``travel the same distance'' (length between outputs), justifying the name \textit{speed}.
If the speed \(\lambda=1\), \(\sigma\) is said to be \textit{parametrized by arc length}.
\definition{
A curve \(\gamma:I\to X\) is called a \textit{geodesic} if it has constant speed and \(L(\gamma)|_{[t,t']}=d(\gamma(t),\gamma(t'))\) for all \(t\leq t'\in I\).
A curve \(\sigma:I\to X\) is called a \textit{local geodesic} if for any \(t\in I\) there exists \(\varepsilon_t>0\) such that \(\sigma|_{I\cap [t-\varepsilon_t,t+\varepsilon_t]}\) is a geodesic.
A space \((X,d)\) is called a \textit{geodesic space} if any two points \(x,y\in X\) have a geodesic \(\gamma_{xy}:[0,1]\to X\) joining them.
Similarly, a space \((X,d)\) is called a \textit{local geodesic space} if any two points \(x,y\in X\) have a local geodesic \(\sigma_{xy}:[0,1]\to X\) joining them.
If such a geodesic is unique, we say that the metric space \((X,d)\) is \textit{uniquely geodesic}.
} %geodesic

Letting \(t=a,t'=b\) as in the definition of a geodesic from $\gamma(a)=x$ to $\gamma(b)=y$ we get \(L(\gamma)|_{[a,b]}=d(x,y)\), showing us that a geodesic can be considered a curve between \(x\) and \(y\) with minimal length.
An alternative definition of a geodesic is the following:
\begin{proposition}
Let \((X,d)\) be a metric space.
A curve \(\gamma:I\to X\) is a geodesic if and only if there exists a constant \(\lambda\geq 0\) such that \(d(\gamma(t),\gamma(t'))=\lambda |t-t'|\) for all \(t,t'\in I\).
\end{proposition}
\begin{proof}
For the first implication, let \(\gamma:I\to X\) be a geodesic.
Note that the definition of a geodesic requires the curve \(\gamma\) to be of constant speed \(\lambda>0\), i.e. \(L(\gamma|_{[t,t']})=\lambda|t-t'|\).
Combining definitions we then get \(d(\gamma(t),\gamma(t'))=L(\gamma|_{[t,t']})=\lambda|t-t'|\) proving the first implication.

Next consider a curve \(\gamma:I\to X\) as in the proposition.
The length of such a curve restricted to any interval \([t,t']\) is
\[L(\gamma|_{[t,t']})
=\sup\sum d(\gamma(t_{i-1}),\gamma(t_i))
=\sup\sum\lambda|t_i-t_{i-1}|
\overset{(1)}{=} \sup \lambda|t-t'|
\overset{(2)}=\lambda|t-t'|=d(\gamma(t),\gamma(t')),\]
where equality (1) follows from the sum being over a partition of the interval \([t,t']\), and equality (2) follows from \(\lambda|t-t'|\) being constant.
This shows that \(\gamma\) has constant speed \(\lambda\) and that \(L(\gamma|_{[t,t']})=d(\gamma(t),\gamma(t'))\), i.e. that \(\gamma\) is a geodesic.

\end{proof}

\begin{proposition}
Every geodesic metric space is a length space.\end{proposition} %Geodesic is length
\begin{proof}
Let \((X,d)\) be a geodesic metric space, and let \(x\neq y\in X\).
By definition of geodesic metric space, there exists a geodesic \(\gamma':[a,b]\to X\) between \(x\) and \(y\) with speed \(\lambda\).
Define \(\varphi:[0,d(x,y)]\to [a,b]\) by \(\varphi(t)=a+\frac{t(b-a)}{d(x,y)}\), and let \(\gamma=\gamma'\circ\varphi\).
We check 
\begin{align*}
    \gamma(0)       &=\gamma'\left(a+\frac{0(b-a)}{d(x,y)}\right)=\gamma'(a)\\
    \gamma(d(x,y))  &=\gamma'\left(a+\frac{d(x,y)(b-a)}{d(x,y)}\right)=\gamma'(b),\\
    d(\gamma(t),\gamma(t')) &=\lambda\left|a+\frac{(b-a)}{d(x,y)}t'-a-\frac{b-a}{d(x,y)}t\right|=\lambda_1|t'-t|,
\end{align*}   
which shows that \(\gamma:[0,d(x,y)]\to X\) is a geodesic with speed \(\lambda_1=\frac{b-a}{d(x,y)}\lambda\) defined on \([0,d(x,y)]\).
By definition of a geodesic \(d(\gamma(0),\gamma(d(x,y))=d(x,y)=\lambda_1|0-d(x,y)|=\lambda_1 d(x,y)\) as \(d(x,y)\geq 0\) by definition of a metric.
As \(x,y\) were chosen to be distinct we further have \(d(x,y)>0\) so we can divide through to get \(\lambda_1=1\).
Next consider \(L(\gamma)=\sup\sum_{i=1}^nd(\gamma(t_i),\gamma(t_{i-1}))=\sup\sum_{i=1}^n |t_i-t_{i-1}|\).
The last equality follows from \(d(\gamma(s),\gamma(t))=|s-t|\) and \(t_i\geq t_{i-1}\).
Now we are summing over a partition of \([0,d(x,y)]\), so \(\sum_i |t_i-t_{i-1}|=d(x,y)-0=d(x,y)\), whereby \(L(\gamma)=d(x,y)\).
By the definition of length, \(L(\sigma)\geq d(x,y)\) for any rectifiable curve \(\sigma\) joining \(x,y\), so 
\[d(x,y)=L(\gamma)=\inf L(\gamma)=d_i(x,y),\] whereby \(d(x,y)=d_i(x,y)\) finishing the proof.\end{proof}

%------------------Hopf-Rinow-----------------
\section{The Hopf--Rinow Theorem}
In the previous section we showed that every geodesic metric space is a length space.
The converse is not always true as seen in the following example:
\begin{example}
Consider the space \(X=\R^2\setminus\{0\}\) equipped with the induced metric \(d_E\) from the Euclidean metric on \(\R^2\).
Any curve \(\sigma:[-t,t]\to X\) between \(x\in X\) and \(-x\in X\) cannot pass through the origin. 
Therefore \(L(\sigma|_{[-t,t]})>d_E(\sigma(-t),\sigma(t))\) whereby \(\sigma\) is not a geodesic. 
\((X,d)\) is therefore not a geodesic metric space.

The infimum of the length of all rectifiable curves between \emph{any} two points \(x,y\in X\), namely \(d_i(x,y)\), is the length of the straight line, i.e. the Euclidean metric \(d_E(x,y)\). 
This can be seen as follows:
If the straight line between \(x\) and \(y\), namely \([x,y]\subseteq\R^2\), does not contain the origin, \(d_i(x,y)=d_E(x,y)\) just as in \((\R^2,d)\). 
In the case where \([x,y]\) \emph{does} contain the origin, let \(P_\varepsilon\) be a point on the unique line passing through \(0\) orthogonal to the line between \(x\) and \(y\) such that \(d(P_\varepsilon,0)=\varepsilon\).
Consider the polygonal curve \(\sigma_\varepsilon=[x,P_\varepsilon,y]\), which is rectifiable.
By Pythagoras this has length \(\sqrt{d(x,0)^2+\varepsilon^2}+\sqrt{\varepsilon^2+d(0,y)^2}\).
As \(\varepsilon>0\) can be chosen arbitrarily small, \(\inf(L(\sigma_\varepsilon))=\sqrt{d(x,0)^2}+\sqrt{d(0,y)^2}=d(x,0)+d(0,y)=d(x,y)\).
As \(d_i(x,y)\geq d(x,y)\) we have \(d(x,y)=\inf(L(\sigma_\varepsilon))=\inf(L(\sigma))=d_i(x,y)\) whereby \(X\) is a length space.
\end{example}

A natural question would now be: under which circumstances are length spaces geodesic spaces?
This is encapsulated in the so-called Hopf--Rinow Theorem. 
Before stating (and proving) this important result, we will need some preliminaries.

\subsection{Continuity of the length function}
\begin{proposition}\label{prop:cont_res}
    Let \(\sigma:[0,1]\to X\) be a rectifiable curve from the interval \([0,1]\) to a metric space \((X,d)\).
    Then \(\length{[0,t]}\) is continuous with respect to the parameter \(t\in [0,1]\).
\end{proposition} %continuous
\begin{proof}
By the definition of \(L\), note that \(\length{[0,b]}-\length{[0,a]}=\length{[a,b]}\), for any $0<a<b<1$.
We now want to show that for every \(\varepsilon>0\) there exists a \(\delta_\varepsilon>0\) such that \(|t-t'|<\delta_\varepsilon\) implies \(\length{[0,t]}-\length{[0,t']}=\length{[t',t]}<\varepsilon\).

\medskip
Indeed, by the definition of a curve \(\sigma\) is continuous on the compact interval \([0,1]\), so it must be uniformly continuous. 
Therefore, for every \(\varepsilon>0\) there exists a \(\delta_\varepsilon>0\) such that \(|t-t_1|<\delta_\varepsilon\) implies \(d(\sigma(t),\sigma(t_1))<\varepsilon\).
As \(L(\sigma)\) is the supremum over all partitions, we can find a partition \(P\) of \([0,1]\) such that 
\begin{equation}\label{eq:L-e}
L(\sigma)-\varepsilon<\sum_{i=1} d(\sigma(t_i),\sigma(t_{i-1})).
\end{equation}
We can now choose a refinement \(P'\) of \(P\) such that \(|t_i-t_{i-1}|<\delta_\varepsilon\) for all \(i\).
By the triangle inequality, the length of the refinement will be greater than or equal to the original partition \(P\), so the inequality \eqref{eq:L-e} still holds.
As we are dealing with a partition of \([0,1]\), we have 
\[L(\sigma)=\sum_{i=1} \length{[t_{i-1},t_i]}\geq \sum_{i=1} d(\sigma(t_i),\sigma(t_{i-1}))\geq L(\sigma)-\varepsilon.\]
Subtracting the different expressions for \(L(\sigma)\) and rearranging we get
\begin{align*}
\sum_{i=1} \left[d(\sigma(t_i),\sigma(t_{i-1}))-\length{[t_{i-1},t]}\right]    &\geq -\varepsilon&\\
\sum_{i=1} \left[\length{[t_{i-1},t_i]}-d(\sigma(t_i),\sigma(t_{i-1}))\right] &\leq \varepsilon.
\end{align*}

As \(\length{[t_{i-1},t_i]}\geq d(\sigma(t_i),\sigma(t_{i-1}))\) all of the terms above are positive.
Selecting only one of the terms must therefore still preserve the inequality, whereby:
\begin{align*}
0\leq\length{[t_{i-1},t_i]}-d(\sigma(t_i),\sigma(t_{i-1}))&\leq \varepsilon\\
\length{[t_{i-1},t_i]}  &\leq \varepsilon+d(\sigma(t_i),\sigma(t_{i-1})).
\intertext{Finally, as \(d(\sigma(t_i),\sigma(t_{i-1}))<\varepsilon\):}
\length{[t_{i-1},t_i]}  &<2\varepsilon,                             
\end{align*}
whenever \(|t_{i-1},t_i|<\delta_\varepsilon\).
As \(\varepsilon\) was chosen arbitrarily that concludes the proof.
\end{proof}

\subsection{Midpoints}
Looking at the length space \(\R^2\setminus\{0\}\) from above we notice that there is no point \emph{exactly} in the middle between \(-1\) and \(1\), because this would be the point \(\{0\}\) which is not in the set.
We would like to know when such midpoints exist.
This motivates the following definition:

\begin{definition}
Let \(x,y\) be two points in a metric space \((X,d)\), and let \(\varepsilon\geq 0\).
An \textit{\(\varepsilon\)-midpoint} between \(x\) and \(y\) is a point \(z\in X\) such that 
\[d(x,z),d(y,z)\leq \frac{d(x,y)}{2}+\varepsilon.\]
A 0-midpoint is simply called a \textit{midpoint}.
If for a given \(\varepsilon>0\) there exists an \(\varepsilon\)-midpoint between any two points, we say that \(X\) \textit{admits \(\varepsilon\)-midpoints}.
If there exists a midpoint between any two points, we simply say that \(X\) \textit{admits midpoints}.
\end{definition} %midpoints
\begin{lemma}[Midpoints]\label{lem:midpoints}
\parenum
\begin{enumerate}
    \item \(X\) is a length space if and only if it admits \(\varepsilon\)-midpoints for all \(\varepsilon>0\).
    \item \(X\) is a geodesic space if and only if it admits midpoints.
\end{enumerate}
\end{lemma} %length if midpoints
\begin{proof}
\parenum
\begin{enumerate}
    \item First assume \(X\) is a length space.
For any two points \(x,y\in X\) and \(\varepsilon>0\) we want to find an \(\varepsilon\)-midpoint.
As \(X\) is a length space, for any rectifiable curve \(\sigma:[0,1]\to X\) with \(\sigma(0)=x,\sigma(1)=y\) we have \(d(x,y)\leq L(\sigma)\).
In particular, there must exist a curve \(\sigma_1\) with length \(d(x,y)\leq L(\sigma_1)\leq d(x,y)+\varepsilon\).
By Proposition \ref{prop:cont_res}, the restricted length function \(\length[\sigma_1]{[0,t]}\) is continuous with respect to the variable \(t\).
Furthermore, \(\length[\sigma_1]{[0,t]}+\length[\sigma_1]{[t,1]}=L(\sigma_1)\), so we can find \(t\in [0,1]\) such that \(\length[\sigma_1]{[0,t]}=\length[\sigma_1]{[t,1]}\) whereby \(L(\sigma_1)=2\length[\sigma_1]{[0,t]}=2\length[\sigma_1]{[t,1]}\geq 2d(x,z), 2d(z,y)\), where \(z=\sigma_1(t)\).
We now have \(2d(z,y)\leq 2L(\sigma_1|_{[0,t]})=L(\sigma_1)\leq d(x,y)+\varepsilon\), whereby 
\begin{align*}
    d(x,z)&\leq\frac{d(x,y)}{2}+\varepsilon/2<\frac{d(x,y)}{2}+\varepsilon,\\
    d(z,y)&\leq\frac{d(x,y)}{2}+\varepsilon/2<\frac{d(x,y)}{2}+\varepsilon,
\end{align*}
showing that \(\sigma_1(t)\) is an \(\varepsilon\)-midpoint and thereby showing that all length spaces admit \(\varepsilon\)-midpoints.

Assume now that a metric space \((X,d)\) admits \(\varepsilon\)-midpoints for all \(\varepsilon>0\).
Let \(x,y\in X\).
Define for a given positive integer \(n\) the number \(\varepsilon_n=\varepsilon/2^{2n}>0\).
Now define \(\sigma(1/2)\) to be an \(\varepsilon_1\)-midpoint between \(x=\sigma(0)\) and \(y=\sigma(1)\).
Next, define \(\sigma(1/4),\sigma(3/4)\) to be \(\varepsilon_2\)-midpoints between \(\sigma(0),\sigma(1/2)\) and \(\sigma(1/2),\sigma(1)\) respectively.
Keep going, defining \(\sigma(k/2^n)\), where \(0<k/2^n<1, k\), with \(k\) odd, to be an \(\varepsilon_n\)-midpoint between \(\sigma\left(\frac{k-1}{2^n}\right)\) and \(\sigma\left(\frac{k+1}{2^n}\right)\).

We now have a function \(\sigma:D\to X\) from the dyadic rationals \(D\), i.e. rationals of the form \(m/2^n\), to the metric space \(X\).
We now claim that \(d\left(\sigma\left(\frac{k}{2^n}\right),\sigma\left(\frac{k+1}{2^{n}}\right)\right)\leq \frac{1}{2^n}(d(x,y)+\varepsilon)\).
First, as \(\sigma\left(\frac{k}{2^n}\sigma\right)\) is an \(\varepsilon_n\)-midpoint between \(\sigma\left(\frac{k-1}{2^n}\right)\) and \(\sigma\left(\frac{k+1}{2^n}\right)\) we have
\begin{align}
d\left(\sigma\left(\frac{k}{2^n}\right),\sigma\left(\frac{k+1}{2^n}\right)\right)
&\leq \frac{1}{2}d\left(\sigma\left(\frac{k-1}{2^n}\right),\sigma\left(\frac{k+1}{2^n}\right)\right)+\varepsilon_n\label{eq:mid1}\\
&=d\left(\sigma\left(\frac{k}{2^n}\right),\sigma\left(\frac{k+1}{2^n}\right)\right)+\varepsilon/2^{2n}.\notag
\end{align}

As \(k\) is odd, we can define \(k-1=2\ell\) to get integers \(\frac{k-1}{2^n}=\frac{\ell}{2^{n-1}}\) and \(\frac{k+1}{2^{n}}=\frac{\ell+1}{2^{n-1}}\).
Either \(\ell\) or \(\ell+1\) are odd.
Assume without loss of generality that \(\ell+1\) is odd.
Then
\begin{align*}
\frac{1}{2}d\left(\sigma\left(\frac{k-1}{2^n}\right),\sigma\left(\frac{k+1}{2^n}\right)\right)
    &=\frac{1}{2}d\left(\sigma\left(\frac{\ell}{2^{n-1}}\right),\sigma\left(\frac{\ell+2}{2^{n-1}}\right)\right)\label{eq:mid2}\\
    &\leq \frac{1}{2^2}d\left(\sigma\left(\frac{\ell}{2^{n-1}}\right),\sigma\left(\frac{\ell+2}{2^{n-1}}\right)\right)+\frac{1}{2}\varepsilon_{n-1}.\notag
\end{align*}

Combining equations \eqref{eq:mid1} and \eqref{eq:mid2}, we get
\begin{align}d\left(\sigma\left(\frac{k}{2^n}\right),\sigma\left(\frac{k+1}{2^n}\right)\right)
&\leq d\left(\sigma\left(\frac{k}{2^n}\right),\sigma\left(\frac{k+1}{2^n}\right)\right)+\varepsilon/2^{2n}\\
&\leq \frac{1}{2^2}d\left(\sigma\left(\frac{\ell}{2^{n-1}}\right),\sigma\left(\frac{\ell+2}{2^{n-1}}\right)\right)+\frac{1}{2}\varepsilon_{n-1}.
\end{align}

Repeating this argument eventually gets us
\begin{align*}
d\left(\sigma\left(\frac{k}{2^n}\right),\sigma\left(\frac{k+1}{2^n}\right)\right)
&\leq \frac{1}{2^n}d(x,y)+\frac{\varepsilon}{2^{2n}}+\frac{1}{2}\frac{\varepsilon}{2^{2(n-1)}}+\ldots+\frac{1}{2^{n-1}}\frac{\varepsilon}{2^2}\\
&\leq \frac{1}{2^n}d(x,y)+\frac{\varepsilon}{2^n}\left(\frac{1}{2^n}+\frac{1}{2^{n-1}}+\ldots+\frac{1}{2}\right)
\end{align*}
Now the parenthesis on the right hand side is a partial sum of the series \(\sum_{i=1}^\infty 1/2^n=1\).
All terms are positive, so the partial sum is strictly smaller than the series, finally giving us
\begin{equation}
d\left(\sigma\left(\frac{k}{2^n}\right),\sigma\left(\frac{k+1}{2^n}\right)\right)\leq 
\frac{1}{2^n}(d(x,y)+\varepsilon)\label{eq:mid3}.    
\end{equation}

Let \(t,t'\in [0,1]\cap D\) be any two dyadic rationals.
Equation \eqref{eq:mid3} then tells us that \(d(\sigma(t),\sigma(t'))=(d(x,y)+\varepsilon)|t-t'|\), showing us that \(\sigma\) is Lipschitz with Lipschitz constant \(d(x,y)+\varepsilon\).
Because \(X\) is complete, and \(\sigma\) is Lipschitz, we can then extend \(\sigma\) continuously to a function \(\sigma:[0,1]\to X\). 
This extension is a curve, where we still have \(d(\sigma(t),\sigma(t'))\leq (d(x,y)+\varepsilon)|t-t'|\), now for all \(t,t'\in [0,1]\).

Let \(P\) be a partition \(0=t_0<t_1<\ldots<t_m=1\) of \([0,1]\).
Using the Lipschitz bound \((d(x,y)+\varepsilon)\) we then get
\begin{align}
\sum_{i=0}^{m-1} d\left(\sigma(t_i),\sigma(t_{i+1})\right)
&\leq\sum_{i=0}^{m-1} (d(x,y)+\varepsilon)|t_i-t_{i+1}|\notag\\
&=(d(x,y)+\varepsilon)\sum_{i=0}^{m-1}|t_i-t_{i+1}|\notag\\
&=d(x,y)+\varepsilon.\label{eq:mid4}
\end{align}
As this holds for all partitions of \([0,1]\) the definition of length gives us \(L(\sigma)\leq d(x,y)+\varepsilon\).
Because \(\varepsilon\) was arbitrary, we get that the infimum of lengths of \emph{all} rectifiable curves \(\sigma'\) must be at most the infimum of curves constructed in this way for any given \(\varepsilon\).
This infimum is \(d(x,y)\), so \(L(\sigma')\leq d(x,y)\) whereby \(X\) is a length space.

\item Remove any mentions of \(\varepsilon\) in the previous proof and proceed in the exact same way.    
\end{enumerate}
\end{proof}

Now we are ready to state the Hopf--Rinow Theorem.
\begin{theorem}[Hopf--Rinow]\label{thm:Hopf_Rinow}
Let \((X,d)\) be a length space.
If \((X,d)\) is complete and locally compact, then
\setenumerate[1]{label=\textbf{(\arabic*)}}
\begin{enumerate}
    \item every closed bounded subset of \((X,d)\) is compact (i.e. \((X,d)\) is proper).
    \item \((X,d)\) is a geodesic space.
\end{enumerate} 
\end{theorem} %Hopf--Rinow
\begin{proof}

\parenum
\begin{enumerate}
\item Let \((X,d)\) be a metric space as in the theorem, and let \(S\) be a closed bounded subset.
We want to show that \(S\) is compact.
As \(S\) is bounded it will be a closed subset of a closed ball with some center \(x\in X\) and some radius \(r\).
Showing that a closed ball in \(X\) is compact will therefore also show that \(S\) is compact.
It is therefore enough to show that closed balls are compact under the assumptions of the theorem.

Let \(x\in X\) be fixed, and denote by \(\rho\) the largest non-negative number \(\rho\) such that the closed ball \(\overline{B}(x,r)=\{y\in X:d(x,y)\leq r\}\) is compact for \(r<\rho\).
As \(X\) is locally compact, \(\rho>0\).
Let \(\overline{B(x,r)}\) be a compact ball with radius \(r\).
As this is compact, for every \(y\in \overline{B(x,r)}\) there is some \(r_y>0\) such that the corresponding closed ball \(B(y,r_y)\) is compact.
The set of all such compact balls cover the compact ball \(\overline{B(x,r)}\), so by compactness there is a finite subset of these which cover \(\overline{B(x,r)}\), i.e. a finite set of \(B(y_i,r_i)\) such that 
\[\overline{B(x,r)}\subsetneq \bigcup_{i=1}^k B(y_i,r_i)\subsetneq \bigcup_{i=1}^k\overline{B(y_i,r_i)}.\]
The middle expression is a finite union of open sets, so it is itself open.
This in turn means that \(X\setminus \bigcup_{i=1}^k B(y_i,r_i)\) is closed by definition.
The right expression is a finite union of compact and closed sets, so it is itself closed and compact.
Together this means that the set
\[Y:=\left(\bigcup_{i=1}^k\overline{B(y_1,r_i)}\right)\cap\left(X\setminus\bigcup_{i=1}^k B(y_i,r_i)\right)
=\bigcup_{i=1}^k\overline{B(y_i,r_i)}\setminus B(y_i,r_i)\]
is closed and compact.

Consider now the map \(X\to [0,\infty]\) given by \(x'\mapsto d(x',\overline{B(x,r)})\), which is continuous on \(X\), and therefore also continuous on the compact subset \(Y\subseteq X\).
As \(Y\) is compact, so is its image \(I\) under \(d(-,\overline{B(x,r)})\), whereby \(I\) is closed.
We now define \(\delta_x=\min \{I\}\subseteq [0,\infty]\).
This is strictly positive as \(\delta_x=0\) would imply that there is a point \(y\in Y\) which is also contained in \(\overline{B(x,r})\), which is a proper subset of the open set \(\bigcup_{i=1}^k B(y_i,r_i)\) and therefore disjoint from \(Y\).
This shows that \(B(x,r+\delta_x)\subseteq \bigcup_{i=1}^k \overline{B(y_i,r_i)}\), so \(\overline{B(x,r+\delta_x)}\) is compact.
We conclude that \(A:=\{r>0:\overline{B(x,r)}\text{ is compact}\}\) is an open set.

Next we want to show that \(\overline{B(x,\rho)}\) is compact if \(\overline{B(x,r)}\) is compact for all \(r<\rho\).
We do this by showing sequential compactness.

Let \((y_j)_{j\in\N}\) be a sequence in \(\overline{B(x,\rho)}\), and let \((\varepsilon_i)_{i\in\N}\) be a sequence of real numbers \(0<\varepsilon_i<\rho\) converging to 0, i.e. \(\varepsilon_i\to 0\) when \(i\to \infty\).

Given \(y_j\) and \(\varepsilon_i\), we want to choose a point \(x_j^i\) such that 
\[x_j^i\in \overline{B(y_j,\varepsilon_i)}\quad\text{ and }\quad x_j^i\in \overline{B(x,\rho-\varepsilon_i/2)}.\]

As \(X\) is a length space there is a continuous curve \(\sigma:[0,1]\to X\) between \(y_j=\sigma(1)\) and \(x=\sigma(0)\) with \(d(x,y_j)\leq L(\sigma)\leq d(x,y_j)+\varepsilon_i/2\).
By Proposition \ref{prop:cont_res} the restricted length function \(\length{[0,t]}\) is continuous in \(t\), so we can choose a \(t\) such that \(L(\sigma|_{[t,1]})=\varepsilon_i\).
Set \(x_j^i=\sigma(t)\).
We now have \(d(y_j,x_j^i)\leq L(\sigma|_{[t,1]})=\varepsilon_i\) so this \(x_j^i\) satisfies the first condition.
Furthermore, 
\[d(x,x_j^i)+\varepsilon_i\leq \length{[0,t]}+\length{[t,1]}=\length{[0,t]}+\varepsilon_i= L(\sigma)\leq d(x,y_j)+\varepsilon_i/2\leq\rho+\varepsilon_i/2.\]
Now \(d(x,x_j^i)+\varepsilon_i\leq\rho+\varepsilon_i/2\) so \(d(x,x_j^i)\leq \rho(x)-\varepsilon/2\), showing that \(x_j^i\) satisfies the second property.

Now as \((x_j^i)_j\in \overline{B(x,\rho-\varepsilon_i/2)}\) which is compact for any \(i\in \N\), \((x_j^1)_{j\in\N}\) has a convergent subsequence \((x_{j_k^1}^1)_{k\in\N}\).
By the same reasoning the sequence \((x_{j_k^1}^2)_{k\in\N}\) has a convergent subsequence \((x_{j_k^2}^2)_{k\in\N}\).
Using the same argument again, the sequence \((x_{j_k^2}^3)_{k\in\N}\) has a convergent subsequence \((x_{j_k^3}^3)_{k\in\N}\), and so on.
For any \(p\in\N\) we can then keep repeating the same argument to get a sequence \((j_k^p)_{k\in\N}\) such that \((x_{j_k^p}^i)_{k\in\N}\) converges for all \(i\leq p\).
A diagonalization argument then gives us a sequence of integers \((j_k)_{k\in\N}\) such that \((x_{j_k}^p)_{k\in\N}\) converges for all \(p\).

Consider now the subsequence \((y_{j_k})_{k\in\N}\) of \((y_j)_{j\in\N}\).
As \((\varepsilon_i)\to 0\) we can choose an \(i\) such that \(\varepsilon_i<\varepsilon\).
Now all the \((x_{j_k}^p)_{k\in \N}\) are convergent for all \(p\), and therefore Cauchy, so by definition \(d(x_{j_k}^i,x_{j_l}^i)\leq \varepsilon\) for sufficiently large \(k,l\).
By the triangle inequality:
\[d(y_{j_k},y_{j_l})\leq d(y_{j_k},x_{j_k}^i)+d(x_{j_k}^i,x_{j_l}^i)+d(x_{j_l}^i,y_{j_l})\leq \varepsilon_i+\varepsilon+\varepsilon_i\leq 3\varepsilon.\]
As \(X\) is assumed complete, the subsequence \((y_{j_k})\) of \((y_j)\) converges, whereby \(\overline{B(x,\rho)}\) is sequentially compact and therefore compact as was to be shown.

Assuming that \(\rho<\infty\) we then have that \(\overline{B(x,\rho)}\) is compact.
But this implies as shown before that there exists a \(\delta>0\) such that \(\overline{B(x,\rho+\delta)}\) is compact which contradicts \(\rho\) being an upper bound.
Therefore \(\rho=\infty\) and we are done with the first part of the theorem.

\item Let \(x,y\in X\).
Since \(X\) is a length space it admits \(\varepsilon\)-midpoints by Lemma \ref{lem:midpoints} for all \(\varepsilon>0\).
Let \((z_j)_{j\in\N}\) be a sequence of \(\frac{1}{j}\)-midpoints between \(x\) and \(y\).
All of these lie in the closed bounded subset \(\overline{B(x,\frac{1}{2}d(x,y)+1)}\), which is compact by part (1). Therefore \((z_j)_{j\in\N}\) has a convergent subsequence.
We show that such a subsequence converges to a midpoint.
Assume for contradiction that the subsequence converges to something that is \emph{not} a midpoint. 
Then there exists a \(N\) such that \(j>N\) implies that \(d(z_j,x)> \frac{1}{2}d(x,y)+1/j\) or \(d(z_j,y)>\frac{1}{2}d(x,y)+1/j\). 
This contradicts \(z_j\) being a \(1/j\)-midpoint for all \(j\), so the subsequence must therefore converge to a midpoint.
By completeness this limit lies in \(X\), whereby \(X\) admits midpoints.
Lemma \ref{lem:midpoints} now proves that \(X\) is a geodesic space.
\end{enumerate}
\end{proof}

An immediate consequence of the Hopf--Rinow Theorem is the following: 
\begin{corollary}
A length space is proper if and only if it is complete and locally compact.
\end{corollary}

%-------------Alexandrov----------------
\chapter{Non-positive curvature}
\section{Non-positive curvature in the sense of Alexandrov}

Given three points \(x,y,z\in X\) in a geodesic space \((X,d)\), we can by definition find geodesics \(\gamma_1,\gamma_2,\gamma_3\) with lengths \(\ell_1,\ell_2,\ell_3\) between those three points, together forming a 
\textit{geodesic triangle} \GTri.
This gives rise to the concept of a \textit{comparison triangle}:
\begin{definition}
Given a geodesic triangle \(\GTri\) in a geodesic space \((X,d)\), a \textit{comparison triangle} is a triangle \(\CTri\) in \(\E^2\) with vertices \(\overline{x},\overline{y},\overline{z}\) having the same side lengths as \GTri.
A point \(p\in [\overline{x},\overline{y}]\) is called a \textit{comparison point} for \(p\in [x,y]\) if \(d(x,p)=d(\overline{x},\overline{p})\), and similarly for points on \([x,z],[y,z]\).
The inner angle of \(\CTri\) at \(\overline{x}\) is called the \textit{comparison angle} between \(y\) and \(z\) at \(x\), denoted \(\overline{\angle}_x(y,z)\), or sometimes \(\angle_{\overline{x}}(\overline{y},\overline{z})\).
\end{definition}

It is important for our following discussion that such a triangle actually exists.
\begin{proposition}
    Every geodesic triangle in a geodesic space \((X,d)\) has a corresponding comparison triangle.
\end{proposition}
\begin{proof}
    Let \(\GTri\) be a geodesic triangle in \((X,d)\) with side lengths \(a,b\) and \(c\).
    Let \([\overline{x},\overline{y}]\) be a line segment in \(\E^2\) with length \(a\).
    Denote by \(C_1\) the circle with center \(\overline{x}\) and radius \(b\) and denote by \(C_2\) the circle with center \(\overline{y}\) and radius \(c\).
    These must intersect at some points \(\overline{z}_1,\overline{z}_2\) due to the triangle inequality.
    Now the triangle with vertices \(\overline{x},\overline{y},\overline{z_1}\) and the triangle with vertices \(\overline{x},\overline{y},\overline{z_2}\) both have side lengths \(a,b\) and \(c\) as desired.    
\end{proof}

With comparison triangles defined we can now define a generalization of the inner angle from Euclidean geometry in \emph{any} geodesic metric space.
\begin{definition}
Let \((X,d)\) be a geodesic metric space, and let \(\gamma_1:[0,1]\to X\) and \(\gamma_2:[0,1]\to X\) be two geodesic paths such that \(\gamma_1(0)=\gamma_2(0)=p\).
The \textit{Alexandrov angle} or the \textit{upper angle} between \(\gamma_1\) and \(\gamma_2\) is defined by
\[\angle(\gamma_1,\gamma_2)=\limsup_{t,t'\to 0}\overline{\angle}_{p}(\gamma_1(t),\gamma_2(t'))=\lim_{\varepsilon\to 0}\sup_{0<t,t'<\varepsilon}\overline{\angle}_{p}(\gamma_1(t),\gamma_2(t'))\leq\pi.\]
The \textit{angular excess} of a geodesic triangle \(\Delta=\GTri\) is defined to be \(\delta(\GTri)=\alpha+\beta+\gamma-\pi\), where \(\alpha,\beta\) and \(\gamma\) are the upper angles at \(x,y\) and \(z\).
\end{definition}
If \(\lim_{t,t'\to 0}\overline{\angle}_{p}(\gamma_1(t),\gamma_2(t'))\) exists we say the angle exists \textit{in the strict sense}.

\begin{definition}\label{def:geo_con}
Let \((X,d)\) be a metric space, and let \(x\in X\).
If \(x\) admits a neighborhood \(U\subseteq X\) such that any two points \(a,b\in U\) can be joined by a geodesic contained in \(U\), we call \(U\) a \textit{geodesically convex neighborhood} of \(x\).
\end{definition}

\begin{figure}\label{fig:alex_angle}
\setlength{\belowcaptionskip}{-10pt}
    \centering
    \includegraphics[width=\textwidth/2]{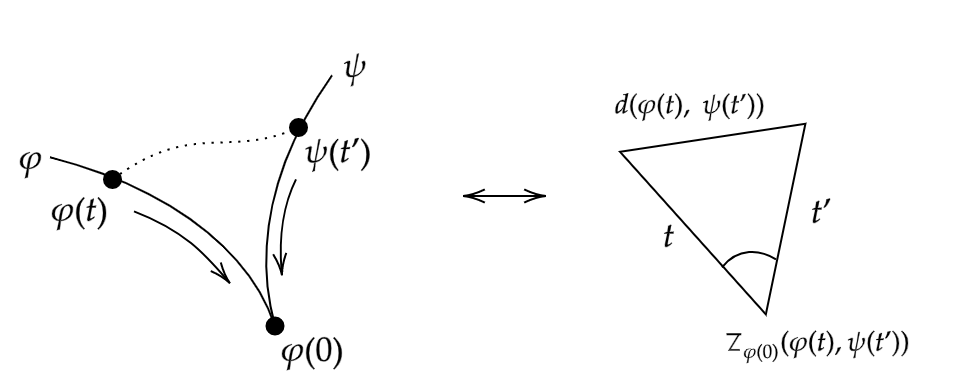}
    \includegraphics[width=\textwidth/2]{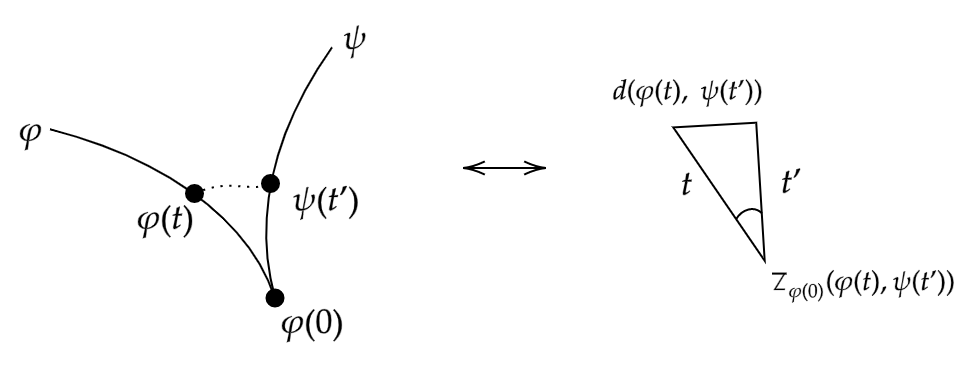}
    \caption{\cite{AngleWeb} The comparison angle changes depending on which points along the geodesics we choose, so we need the Alexandrov angle to give one value. 
    Here the geodesics are labelled \(\varphi,\psi\) instead of \(\gamma_1,\gamma_2\).}
\end{figure}
We are now ready to state the main definition of this section:
\begin{definition}
Let \((X,d)\) be a metric space.
If every point \(x\in X\) has a geodesically convex neighborhood \(U\) such that the angular excess of any geodesic triangle with edges in \(U\) is non-positive, \(X\) is said to be \textit{non-positively curved in the sense of Alexandrov}.
\end{definition}
\begin{remark}
In \(\E^2\), the Alexandrov angle and the usual inner angle coincide.
\end{remark}

The angle function \((\gamma_1,\gamma_2)\mapsto\angle(\gamma_1,\gamma_2)\) is obviously symmetric and positive.
However, it is possible that two geodesics are identical in a sufficiently small neighborhood around \(p\) in which case \(\angle(\gamma_1,\gamma_2)=0\).
This shows that \(\angle\) is not positively definite, so \(\angle\) is not a metric.
The following Proposition however shows that the triangle inequality holds, whereby \(\angle\) \emph{is} a \textit{pseudometric}:

\begin{proposition}\label{prop:pseudo}
Let \((X,d)\) be a metric space with three geodesics \(\gamma_1,\gamma_2,\gamma_3\) all issuing from a point \(p=\gamma_1(0)=\gamma_2(0)=\gamma_3(0)\).
Then the triangle inequality holds:
\[\angle(\gamma_1,\gamma_2)\leq \angle(\gamma_1,\gamma_3)+\angle(\gamma_3,\gamma_2).\]
\end{proposition}
\begin{proof}
Assume for contradiction that \(\angle(\gamma_1,\gamma_2)>\angle(\gamma_1,\gamma_3)+\angle(\gamma_3,\gamma_2)\).
We can then choose \(\delta\) such that \(\angle(\gamma_1,\gamma_2)>\angle(\gamma_1,\gamma_3)+\angle(\gamma_3,\gamma_2)+3\delta\).
From properties of limsup we get the existence of an \(\varepsilon\) such that:
\parenum
\begin{enumerate}
    \item \(\overline{\angle}_p(\gamma_1(t_1),\gamma_3(t_3))<\angle(\gamma_1,\gamma_3)+\delta\) for all \(0<t_1,t_3<\varepsilon\),
    \item \(\overline{\angle}_p(\gamma_2(t_2),\gamma_3(t_3))<\angle(\gamma_2,\gamma_3)+\delta\) for all \(0<t_2,t_3<\varepsilon\), and
    \item \(\overline{\angle}_p(\gamma_1(t_1),\gamma_2(t_2))>\angle(\gamma_1,\gamma_2)-\delta\) for some \(0<t_1,t_2<\varepsilon\).
\end{enumerate}
Choose \(t_1,t_2\) such that (3) holds, and choose \(\alpha\) such that \(\overline{\angle}_p(\gamma_1(t_1),\gamma_2(t_2))>\alpha >\angle(\gamma_1,\gamma_1)-\delta\).

Now construct a triangle \(\overline{\Delta}(x_1,y,x_2)\)in \(\E^2\) such that the sides \([x_1,y]\) and \([y,x_2]\) meet at \(y\) with an inner angle \(0<\alpha<\pi\) with sidelengths \(t_1\) and \(t_2\), respectively.

As \(0<\alpha<\pi\), the triangle is non-degenerate.
Combining choices of \(\alpha\) and \(\delta\):
\begin{align*}
\alpha  &>\angle(\gamma_1,\gamma_2)-\delta          &\text{by choice of } \alpha\\
        &>\angle(\gamma_1,\gamma_3)+\angle(\gamma_2,\gamma_3)+3\delta-\delta    &\text{by choice of }\delta\\
        &=\angle(\gamma_1,\gamma_2)+\angle(\gamma_1,\gamma_3)+2\delta.
\end{align*}

Therefore we can choose a point \(x_3\) on the side \([x_1,x_2]\) splitting the triangle into two smaller triangles with the line \([x_3,y]\), such that the interior angles \(\alpha_1,\alpha_2\) between \([x_1,y],[y,x_3]\) and between \([x_2,y],[y,x_3]\), respectively, satisfy \(\alpha_1>\angle(\gamma_1,\gamma_3)+\delta\) and \(\alpha_2>\angle(\gamma_2,\gamma_3)+\delta\).
Denote by \(t\) the length of \([y,x_3]\), which satisfies \(t<\max\{t_1,t_2\}\) implying \(t<\varepsilon\).
Using inequalities (1) and (2) we get 
\begin{align*}
\overline{\angle}_p(\gamma_1(t_1),\gamma_3(t))    &<\angle(\gamma_1,\gamma_3)+\delta<\alpha_1\\
\overline{\angle}_p(\gamma_2(t_2),\gamma_3(t))   &<\angle(\gamma_2,\gamma_3)+\delta<\alpha_2.
\end{align*}
As the angle in the comparison triangles are smaller, the distances in the comparison triangle (and thereby also the distances in \(X\)) must be smaller, so \(d(\gamma_3(t),\gamma_2(t_2))<|x_3-x_2|\) and \(d(\gamma_1(t_1),\gamma_3(t))<|x_3-x_1|\).
This combined with \(\overline{\angle}_p(\gamma_1(t_1),\gamma_2(t_2))>\alpha\) yields
\begin{align*}
d(\gamma_2(t_2),\gamma_1(t_1))  &>|x_1-x_2|\\
                    &=|x_3-x_1|+|x_3-x_2|\\
                    &>d(\gamma_3(t),\gamma_2(t_2))+d(\gamma_1(t_1),\gamma_3(t)),
\end{align*}
contradicting the triangle inequality of the metric \(d\) and so finishing the proof.
\end{proof}

To end off this section we will state and prove Alexandrov's Lemma.
\begin{lemma}[Alexandrov's Lemma]\label{lem:alex}
    Let \((X,d)\) be a geodesic metric space, and let \(p,x,y,z\) be distinct points in \(X\) with \(z\in [x,y]\).
    Then the expressions
    \begin{align*}
        &\overline{\angle}_x([x,p],[x,y])-\overline{\angle}_x([x,p],[x,z])\\
        &\overline{\angle}_z([z,p],[z,x])+\overline{\angle}_z([z,p],[z,y])-\pi
    \end{align*}
    have the same sign.
\end{lemma}
\begin{proof}
Extend the comparison line \([\overline{x},\overline{z}]\) to a line \([\overline{x},\tilde{y}]\) such that \(d(x,y)=d(\overline{x},\tilde{y})\), whereby \(d(\overline{x},\overline{z})=d(x,z)\) and \(d(\overline{z},\tilde{y})=d(z,y)\).
Now as increasing an angle makes the opposite side longer and vice versa, the following expressions must have the same sign:
\begin{align*}
    \angle_{\overline{x}}([\overline{x},\overline{p}],[\overline{x},\overline{y}])
            &-  \angle_{\overline{x}}([\overline{x},\overline{p}],[\overline{x},\tilde{y}])\\
    d(p,y)  &-  d(\overline{p},\tilde{y})\\
    \angle_{\overline{z}}([\overline{z},\overline{p}],[\overline{z},\overline{y}])  &-  \angle_{\overline{z}}           ([\overline{z},\overline{p}],[\overline{z},\tilde{y}]).
\end{align*}
Because of how we chose \(\tilde{y}\) we have \(\angle_{\overline{x}}([\overline{x},\overline{p}],[\overline{x},\tilde{y}])=\angle_{\overline{x}}([\overline{x},\overline{p}],[\overline{x},\overline{z}])\).
As we are in \(\E^2\) also \(\angle_{\overline{z}}([\overline{z},\overline{p}],[\overline{z},\tilde{y}]=\pi-\angle_{\overline{z}}([\overline{z},\overline{p}],[\overline{z},\overline{x}])\), finishing the proof.
\end{proof}

%------------------Busemann------------------
\section{Non-positive curvature in the sense of Busemann}
Non-positive curvature in the sense of Alexandrov compares the angles of geodesic triangles to the angles in their corresponding comparison triangles.
Another approach to non-positive curvature would be to instead compare the length between midpoints of geodesic triangles to the length between midpoints in their corresponding comparison triangles.
This is formalized in the following definition:
\begin{definition}\label{def:neg_buse}
Let \((X,d)\) be a metric space.
\(X\) is said to have \textit{non-positive curvature in the sense of Busemann} if every point \(x\in X\) has a geodesically convex neighborhood \(U\) such that for any geodesic triangle \(\GTri[a][b][c]\) with edges in \(U\) we have 
\[d(m,m')\leq\frac{1}{2}d(b,c)=d(\overline{m},\overline{m}'),\]
where \(m\) and \(m'\) are the midpoints of \([a,b]\) and \([a,c]\), respectively, with comparison midpoints \(\overline{m},\overline{m}'\).
\end{definition}

An equivalent definition looks at the relation between \emph{any} two geodesics parametrized proportional to the arc length in the same neighborhood \(U\) of \(x\):
\begin{proposition}\label{prop:conv_mid}
Let \((X,d)\) be a metric space such that any \(x\in X\) is contained in a geodesically convex neighborhood \(U_x\subseteq X\). The following are equivalent:
\parenum
\begin{enumerate}
    \item Any geodesic triangle \(\Delta(a,b,c)\) in \(U_x\) satisfies \(d(m,m')\leq\frac{1}{2}d(b,c)\), where \(m\) and \(m'\) are the midpoints of \([a,b]\) and \([a,c]\) respectively.
    \item Each pair of geodesics \(\gamma_1:[0,1]\to U_x,\gamma_2:[0,1]\to U_x\) parametrized proportional to the arc length satisfy \(d(\gamma_1(t),\gamma_2(t))\leq (1-t)d(\gamma_1(0),\gamma_2(0))+td(\gamma_1(1),\gamma_2(1))\) for all \(t\in[0,1]\).
\end{enumerate}
\end{proposition}
\begin{proof}
\((1)\implies (2)\):
We will prove the statement for all dyadic rationals in \([0,1]\) by induction.
By continuity of \(\gamma_1\) and \(\gamma_2\), we will then have proved it for all real numbers in \([0,1]\).

\textbf{Induction start:}
The first dyadic number for which we need to show the stament is \(t=1/2\).
Consider the geodesic \(\gamma_3:[0,1]\to U_x\) with \(\gamma_3(0)=\gamma_1(0)\) and \(\gamma_3(1)=\gamma_2(1)\).
By the triangle inequality we have \(d(\gamma_1(1/2),\gamma_2(1/2))\leq d(\gamma_1(1/2),\gamma_3(1/2))+d(\gamma_3(1/2),\gamma_2(1/2))\).
Condition (1) gives applied to the two triangles \(\Delta(\gamma_1(0),\gamma_1(1),\gamma_2(1))\) and \(\Delta(\gamma_1(0),\gamma_2(1),\gamma_2(0))\) gives us \(d(\gamma_1(1/2),\gamma_3(1/2))\leq \frac{1}{2}d(\gamma_1(1),\gamma_2(1))\) and \(d(\gamma_3(1/2),\gamma_2(1/2))\leq \frac{1}{2}d(\gamma_1(0),\gamma_2(0))\).
Combining finally gives us
\begin{align*}
d(\gamma_1(1/2),\gamma_2(1/2))&\leq d(\gamma_1(1/2),\gamma_3(1/2))+d(\gamma_3(1/2),\gamma_2(1/2)\\
&\leq \left(1-\frac{1}{2}\right)d(\gamma_1(0),\gamma_2(0))+\frac{1}{2}d(\gamma_1(1),\gamma_2(1)),\end{align*}
showing that (1) implies (2) for \(t=1/2\).

\textbf{Induction step:}
Assume (1), and assume that (2) holds for all dyadic rationals of the form \(\frac{k}{2^n}\), where \(k\in\{0,1,\ldots,2^n\}\).
We then need to show that these imply (2) for any \(t\in\left\{\frac{\ell}{2^{n+1}}:\ell\in\{0,1,\ldots,2^{n+1}\}\right\}\).

Let \(\ell\) be odd, whereby \(\ell-1\) and \(\ell+1\) are even.
In particular, \(\frac{\ell-1}{2^{n+1}}\) and \(\frac{\ell+1}{2^{n+1}}\) are of the form \(\frac{k}{2^n}\).
Now \(\frac{\ell}{2^{n+1}}\) is exactly halfway between the dyadic rational numbers \(\frac{\ell-1}{2^{n+1}}\) and \(\frac{\ell+1}{2^{n+1}}\) for which (2) holds by the induction hypothesis.
Consider now the geodesic \(\gamma_3:[0,1]\to U_x\) from \(\gamma_1\left(\frac{\ell-1}{2^{n+1}}\right)\) to \(\gamma_2\left(\frac{\ell+1}{2^{n+1}}\right)\).
The triangle inequality gives us
\[d\left(\gamma_1\left(\frac{\ell}{2^{n+1}}\right),\gamma_2\left(\frac{\ell}{2^{n+1}}\right)\right)
\leq
d\left(\gamma_1\left(\frac{\ell}{2^{n+1}}\right),\gamma_3\left(\frac{1}{2}\right)\right)
+
d\left(\gamma_3\left(\frac{1}{2}\right),\gamma_2\left(\frac{\ell}{2^{n+1}}\right)\right),\]
while (1) gives us
\begin{align*}
d\left(\gamma_1\left(\frac{\ell}{2^{n+1}}\right),\gamma_3\left(\frac{1}{2}\right)\right)
&\leq
\frac{1}{2}d\left(\gamma_1\left(\frac{\ell+1}{2^{n+1}}\right),\gamma_2\left(\frac{\ell+1}{2^{n+1}}\right)\right)\\
d\left(\gamma_3\left(\frac{1}{2}\right),\gamma_2\left(\frac{\ell}{2^{n-1}}\right)\right)
&\leq
\frac{1}{2}d\left(\gamma_1\left(\frac{\ell-1}{2^{n+1}}\right),\gamma_2\left(\frac{\ell-1}{2^{n+1}}\right)\right)
\end{align*}
Combining finally gets us (2), finishing the induction step.

\((2)\implies (1)\):
Let \(\gamma_1\) be the geodesic from \(a=\gamma_1(0)\) to \(b=\gamma_1(1)\), and let \(\gamma_2\) be the geodesic from \(a=\gamma_2(0)\) to \(c=\gamma_2(1)\).
The midpoints are then \(m=\gamma_1(1/2),m'=\gamma_2(1/2)\).
(2) then tells us
\[d(m,m')\leq\left(1-\frac{1}{2}\right)d(\gamma_1(0),\gamma_2(0))+\frac{1}{2}d(\gamma_1(1),\gamma_2(1))=\frac{1}{2}d(a,a)+\frac{1}{2}d(b,c)=\frac{1}{2}d(b,c),\]
which is what was to be shown.
\end{proof}

The second of the equivalent statements in Proposition \ref{prop:conv_mid} is the definition of \(d\) being a convex function when restricted to \(U_x\).
This justifies saying that the induced metric on \(U_x\) from \((X,d)\) is convex if one of the two statements holds.

\begin{definition}
\label{def::loc_conv}
    Let $(X,d)$ be a metric space. 
    If every \(x\in X\) has a geodesically convex neighborhood where the induced metric is convex, we call \((X,d)\) \textit{locally convex}.
\end{definition}

Comparing all definitions we get that \((X,d)\) is non-positively curved in the sense of Busemann if and only if \((X,d)\) is locally convex.

%--------------CAT(0) ineq-------------
\chapter{CAT($0$) spaces}
\section{The CAT($0$)-inequality}
\begin{definition}
Let \((X,d)\) be a metric space, let \(\Delta=\Delta(x,y,z)\) a geodesic triangle in \(X\), and let \(\overline{\Delta}\subseteq\E^2\) its corresponding comparison triangle.
\(\Delta\) is said to satisfy the \emph{CAT($0$)-inequality} if we for all \(p,q\in\Delta\) with comparison points \(\overline{p},\overline{q}\in\overline{\Delta}\) have \(d(p,q)\leq d(\overline{p},\overline{q})\).
A \textit{CAT($0$)-space} is a geodesic metric space \((X,d)\) such that every geodesic triangle in \(X\) satisfies the CAT($0$)-inequality.
\end{definition}

If a CAT($0$)-space is furthermore complete we may call it a \textit{Hadamard-space}.
\begin{figure}[h]
    \centering
    \includegraphics[width=2\textwidth/3]{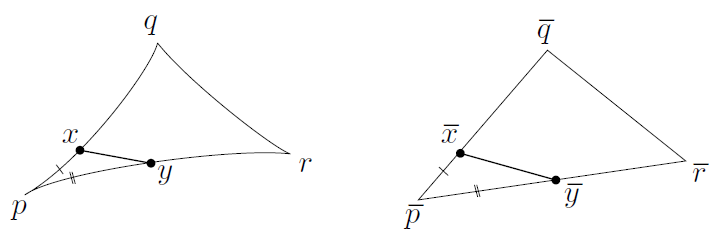}
    \caption{\cite{BridHaef} The CAT($0$)-inequality. The distance between the points \(x\) and \(y\) is no greater than the distance between their comparison points \(\overline{x}\) and \(\overline{y}\)}
    \label{fig:ineq}
\end{figure}

\begin{definition}\label{def:loc_cat}
A metric space \((X,d)\) is said to be \textit{non-positively curved} or \textit{locally} CAT($0$) if for every \(x\in X\) there exists an \(r_x>0\) such that \(B(x,r_x)\) is a CAT($0$)-space.
\end{definition}
\begin{remark}
As the CAT($0$) inequality holds for all points, it must in particular hold for midpoints.
All CAT($0$) spaces are therefore negatively curved in the sense of Busemann as defined in Definition \ref{def:neg_buse}.
\end{remark}

The ``$0$'' in ``CAT($0$)'' is due to the curvature being \(\leq 0\) everywhere.
If instead of using \(\E^2\) we used other \textit{model spaces}, such as hyperbolic space, it would be possible to define spaces with curvature \(\leq \kappa\), fittingly called CAT(\(\kappa\))-spaces.
We will however restrict ourselves to CAT($0$) spaces.

\begin{proposition}\label{prop:apmid_close}
Let \((X,d)\) be a CAT($0$)-space.
\parenum
\begin{enumerate}
    \item \(X\) is uniquely geodesic;
    \item Every local geodesic in \(X\) is a geodesic;
    \item \(X\) is contractible;
    \item Approximate midpoint are close to midpoints in the following sense:
    For every \(\varepsilon>0\) and \(0<L<\infty\) there exists some \(\delta=\delta(\varepsilon,L)\) with the following property: \(d(m,m')<\varepsilon\) for every \(\delta\)-midpoint \(m'\) between \(x\) and \(y\), where \(m\) is the midpoint between \(x\) and \(y\), and \(d(x,y)\leq L\).
\end{enumerate}
\end{proposition}
\begin{proof}
\parenum
\begin{enumerate}
    \item Assume for contradiction that there exists two different geodesics \([x,y]\) and \([x,y]'\) in \(X\) connecting the points \(x\) and \(y\).
Choose \(z\in[x,y]\) and \(z'\in[x,y]'\) such that \(d(x,z)=d(x,z')\) and \(z\neq z'\).
Such a pair must exist as \([x,y]\) and \([x,y]'\) are assumed distinct.
Now consider the comparison triangle of the geodesic triangle \([x,z],[z,y],[x,y]'\).
We have \(d(x,z)+d(z,y)=d(x,y)\) so the comparison triangle must be degenerate, whereby the CAT($0$)-inequality gives us \(d(z,z')\leq d(\overline{z},\overline{z}')=0\), contradicting \(z\) and \(z'\) being distinct.
    \item Let \(\sigma:[a,b]\to X\) be a local geodesic, and define \(S=\{t\in [a,b]:\sigma|_{[a,t]}\text{ is a geodesic}\}\).
First we show that \(S\) is closed.
By hypothesis, \(S\) contains its infimum \(a\), as well as a neighborhood of \(a\).
Take \(s:=\sup\{t\in S\}\).
We just need to show \(s\in S\).

Let \(t\leq s\). 
In Proposition \ref{prop:cont_res} we showed that \(L(\sigma|_{[a,t]})\) is continuous as a function of \(t\), so the limit \(\lim_{t\to s}\length{[a,t]}\) exists and equals \(L(\sigma|_{[a,s]})\).
The restriction \(\sigma|_{[a,t]}\) is a geodesic by definition of \(s\), so \(d(\sigma(a),\sigma(t))=L(\sigma|_{[a,t]})=\lambda|t-a|\) for some constant \(\lambda\geq 0\).
Taking the limit as \(t\to s\) we get
\[d(\sigma(a),\sigma(s))=\lim_{t\to s}d(\sigma(a),\sigma(t))=\lim_{t\to s}\length{[a,t]}=L(\sigma|_{[a,s]})=\lambda|t-s|,\]
whereby \(\sigma|_{[a,s]}\) is a geodesic and \(s\) therefore in \(S\), showing that \(S\) is closed.
 
Next we show that \(S\) is open.
Let \(t\in S\).
As \(\sigma\) is a local geodesic, there exists a neighborhood \([t-\varepsilon,t+\varepsilon]\) of \(t\), with \(\varepsilon>0\), such that \(\sigma|_{[t-\varepsilon,t+\varepsilon]}\) is a geodesic.
Let \(t-\varepsilon<p<t<q<t+\varepsilon\), and consider the geodesic triangle \(\Delta([\sigma(a),\sigma(t)],[\sigma(t),\sigma(q)],[\sigma(q),\sigma(a)])\).
We claim that the comparison triangle \(\overline{\Delta}\) is degenerate.
Assume not.
Then
\begin{align*}
d(\overline{\sigma(p)},\overline{\sigma(t)})+d(\overline{\sigma(t)},\overline{\sigma(q)})
    &>d(\overline{\sigma(p)},\overline{\sigma(q)})\tag*{(as non-degenerate)}\\
    &\geq d(\sigma(p),\sigma(q))\tag*{(CAT(0) inequality)}\\
    &=d(\sigma(p),\sigma(t))+d(\sigma(t),\sigma(q))\tag*{(as local geodesic)}\\
    &=d(\overline{\sigma(p)},\overline{\sigma(t)})+d(\overline{\sigma(t)},\overline{\sigma(q)})
\end{align*}
which is a contradiction.
\(\overline{\Delta}\) must therefore be degenerate.
This in turn means \(d(\sigma(a),\sigma(t))+d(\sigma(t),\sigma(q))=d(\sigma(a),\sigma(q))\) whereby \(\sigma|_{[a,q]}\) is a geodesic.
In other words, for every \(t\in S\) there exists a \(q>t\) where \(q\in S\), showing that \(S\) is open.
Now \(S\subseteq [a,b]\) is both open and closed.
The only subsets of \([a,b]\) which are both open and closed are the empty set and the entire set, as \([a,b]\) is connected.
As we have seen above, \(S\) is non-empty as \(\sigma\) is a local geodesic, so \(S=[a,b]\) whereby \(\sigma\) is a geodesic.

    \item Let \(x\in X\) be a point. 
Consider the map \(c:X\times [0,1]\to X\) which maps \((y,t)\) to the point \(z\in X\) on the geodesic segment \([x,y]\subseteq X\), such that \(d(x,z)=td(x,y)\).
Such a point exists as geodesics exist and are unique by part (1), and vary continuously with their endpoints as seen in Proposition \ref{prop:cont_res}.
We now see that \(c\) is a deformation retraction of \(X\) to \(x\), so \(X\) is contractible.

    \item We start by considering only Euclidean space, i.e. \(X=\E^2\).
Let \(\varepsilon>0\), and let \(0<L<\infty\).
Let \(x,y\in X\) be such that \(d(x,y)=L\), and let \(m\) be their midpoint.
We want to find a \(\delta(\varepsilon,L)>0\) such that the proposition is true.
Draw the line segment \([x,y]\) between \(x\) and \(y\), and draw two circles \(S_x,S_y\) with radius \(L/2+\delta\) centered around \(x\) and \(y\), respectively.
These have radius bigger than half the distance between the two points, so they intersect at points \(m_1\) and \(m_2\).
For \(\delta\) small enough, \(m_1\) and \(m_2\) are the points in \(\overline{B(x,L/2+\delta)}\cap \overline{B(y,L/2+\delta)}\) with maximal distance to \(m\).
These points are exactly the \(\delta\)-midpoints between \(x\) and \(y\).
Now the line segments \([m,m_i]\) for \(i\in\{1,2\}\) are orthogonal to the line segment \([x,y]\).
Applying Pythagoras' theorem we get
\begin{align*}
&d(m,m_i)^2+d(m,y)^2=d(m_i,y)^2\\
&\iff d(m,m_i)^2+(L/2)^2=(L/2+\delta^2)\\
&\iff d(m,m_i)^2=\delta(L+\delta).
\end{align*}
We conclude that we want to find a \(\delta\) such that \(\sqrt{\delta(L+\delta)}<\varepsilon\).
As the left hand side can be arbitrarily small by choosing \(\delta\) small enough, we can find infinitely many such \(\delta\).
In the case where \(d(x,y)<L\), then a similar calculation gives us 
\[d(m,m_i')=\sqrt{\delta(d(x,y)+\delta)}<\sqrt{\delta(L+\delta)}<\varepsilon.\]

Now we consider general CAT($0$) spaces.
We again let \(x,y\in X\) be points such that \(d(x,y)\leq L\), and let \(m'\) be a \(\delta\)-midpoint, where we use the same \(\delta=\delta(\varepsilon,L)\) as in the Euclidean case.
Consider now the triangle \(\Delta(x,y,m')\) and its corresponding comparison triangle \(\overline{\Delta}(\overline{x},\overline{y},\overline{m'})\).
The CAT($0$) inequality then finally gives us
\[d(m,m')\leq d(\overline{m},\overline{m'})<\varepsilon,\]
finishing the proof.
\end{enumerate}
\end{proof}

%-----------------equivalent definintions----------------
\section{Equivalent definitions of CAT(0)-spaces}
There are multiple other ways of defining CAT(0)-spaces.
\begin{proposition}\label{prop:equiv}
Let \((X,d)\) be a geodesic metric space.
The following are equivalent:
\parenum
\begin{enumerate}
    \item \(X\) is a CAT(0)-space;
    \item For every geodesic triangle \(\Delta=\Delta(x,y,z)\subseteq X\) and every point \(p\in [y,z]\) with comparison point \(\overline{p}\in [\overline{y},\overline{z}]\) we have \(d(x,p)\leq d(\overline{x},\overline{p})\);
    \item For every geodesic triangle \(\Delta=\Delta(x,y,z)\subseteq X\) and every pair of points \(p\in (x,y],q\in (x,z]\), the angles at \(\overline{x}\) in the comparison triangle satisfies \(\angle_{\overline{x}}(\overline{p},\overline{q})\leq\angle_{\overline{x}}(\overline{y},\overline{z})\);
    \item The Alexandrov angle between sides of any geodesic triangle in \(X\) with distinct vertices is less than or equal to the angle between the sides in the comparison triangle in \(\E^2\). 
\end{enumerate}
\end{proposition}
\begin{proof}
(2) is a special case of the CAT(0)-inequality, so it follows trivially from (1).
(4) also follows immediately from (3) as the Alexandrov angle is no greater than any comparison angle.

\textbf{(1) is equivalent to (3):}
Let \(x,y,z,p,q\) be as in (3).
Construct the comparison triangles \(\overline{\Delta}=\overline{\Delta}(\overline{x},\overline{y},\overline{z})\) and \(\overline{\Delta}'=\overline{\Delta}(\overline{x}',\overline{p}',\overline{q}')\) with interior angles \(\angle_{\overline{x}}(\overline{y},\overline{z})=\alpha,\angle_{\overline{x}}(\overline{p},\overline{q})=\alpha'\) at \(\overline{x}\) and \(\overline{x}'\), respectively.
By the law of cosines, \(d(\overline{p}',\overline{q}')=d(p,q)\leq d(\overline{p},\overline{q})\) is true if and only if \(\alpha\geq \alpha'\).
The first inequality is the CAT(0)-inequality, and the second is the inequality \(\angle_{\overline{x}}(\overline{p},\overline{q})\leq\angle_{\overline{x}}(\overline{y},\overline{z})\), the same as in (3).
Therefore (1) is equivalent to (3).

\textbf{(2) implies (3):}
Let \(\overline{\Delta}''=\overline{\Delta}(\overline{x}'',\overline{y}'',\overline{q}'')\) be another comparison triangle of \(\Delta(x,y,q)\) with interior angle \(\alpha''\) at \(\overline{x}''\).
By (2) we have \(d(\overline{y}'',\overline{q}'')=d(y,q)\leq d(\overline{y},\overline{q})\).
Similarly as before, the law of cosines now gives us \(\alpha''\leq \alpha\).
Applying (2) to \(\overline{\Delta}''\) instead gives us \(d(\overline{p}',\overline{q}')=d(p,q)\leq d(\overline{p}'',\overline{q}'')\), where \(d(\overline{p}',\overline{q}')\) is the distance between the vertices \(\overline{p}'\) and \(\overline{q}'\) of \(\overline{\Delta}'\), and where \(\overline{p}''\) is the comparison point of \(p\) on the side \([\overline{x}'',\overline{y}'']\).
Using the law of cosines yet again gives \(\alpha'\leq\alpha''\).
Combining inequalities gets us the desired \(\alpha'\leq\alpha\).

\textbf{(3) implies (4):}
Immediate consequence of (3).

\textbf{(4) implies (2):}
Let \(\Delta(x,y,z)\) be a geodesic triangle and let \(p\in[y,z]\).
The geodesic segment from \(x\) to \(p\) forms two geodesic triangles.
Let \(\gamma,\gamma'\) be the Alexandrov angles \([x,p]\) makes at \(p\) with \([y,p]\) and \([p,z]\), respectively, and let \(\overline{\beta}\) be the angle at \(\overline{y}\) in the comparison triangle \(\overline{\Delta}(\overline{x},\overline{y},\overline{z})\).
Draw the two comparison triangles \(\overline{\Delta}(\overline{x'},\overline{y'},\overline{p'})\) of \(\Delta(x,y,p)\) and \(\overline{\Delta}(\overline{x}',\overline{p}',\overline{z}')\) of \(\Delta(x,p,z)\).
These share the side \([\overline{x}',\overline{p}']\), so we can superimpose the two to get a quadrilateral \([\overline{x}',\overline{y}',\overline{p}',\overline{z}']\).
Denote by \(\overline{\beta}'\) the angle at \(\overline{y}'\), and let \(\overline{\gamma},\overline{\gamma}'\) be the angles at \(\overline{p}'\) similarly to \(\gamma,\gamma'\).
By Proposition \ref{prop:pseudo} the Alexandrov angle is a pseudometric, so \(\gamma+\gamma'\geq\angle(y,z)=\pi\). 
By (4) we then have \(\overline{\gamma}+\overline{\gamma'}\geq\pi\), or more suggestively, \(\overline{\gamma}+\overline{\gamma'}-\pi\geq 0\).
Alexandrov's Lemma (\ref{lem:alex}) states that \(\overline{\beta}-\overline{\beta}'\) has the same sign as \(\overline{\gamma}+\overline{\gamma}'-\pi\), which was just shown to be positive.
We therefore have \(\overline{\beta}'\leq\overline{\beta}\), whereby the law of cosines gives us \(d(x,p)=d(\overline{x}',\overline{p}')\leq d(\overline{x},\overline{p})\) as desired.
\end{proof}

%---------------four point-----------------
\section{The four point condition}
So far we have used triangles for all our considerations.
One can instead reformulate CAT(0) spaces in terms of quadrilaterals, using the so called \emph{CAT(0) 4-point condition}.
\begin{definition}
Let \((X,d)\) be a (pseudo)-metric space and let \((x_1,x_2,y_1,y_2)\) be a four-tuple of points in \(X\).
A \emph{subembedding} in \(\E^2\) of \((x_1,x_2,y_1,y_2)\) is a four-tuple of points \((\overline{x}_1,\overline{x}_2,\overline{y}_1,\overline{y}_2)\) in \(\E^2\) such that
\begin{align}
    d(x_i,y_j)  &=d(\overline{x}_i,\overline{y}_j)\quad\text{for }i,j\in\{1,2\}\label{quad:dif}\\
    d(x_1,x_2)  &\leq d(\overline{x}_1,\overline{x}_2)\label{quad:x}\\
    d(y_1,y_2)  &\leq d(\overline{y}_1,\overline{y}_2)\label{quad:y}.
\end{align}
\((X,d)\) is said to satisfy the \textit{CAT(0) 4-point condition} if every four-tuple of points in \(X\) has a subembedding in \(\E^2\).
\end{definition}
\begin{proposition}
    Let \((X,d)\) be a complete metric space.
    Then the following are equivalent:
    \begin{enumerate}
        \item \(X\) is a CAT(0) space;
        \item \(X\) satisfies the CAT(0) four-point condition and every pair of points \(x,y\in X\) has approximate midpoints.
    \end{enumerate}
\end{proposition}
\begin{proof}
\textbf{(1) implies (2):}
Every CAT(0) space is uniquely geodesic, so midpoints exist.
Let \((x_1,x_2,y_1,y_2)\) be a four-tuple in \(X\).
Consider the two comparison triangles \(\overline{\Delta}(\overline{x}_1,\overline{x}_2,\overline{y}_1)\) and \(\overline{\Delta}(\overline{x}_1,\overline{x}_2,\overline{y}_2)\).
Superimposing these along the common edge \([\overline{x}_1,\overline{x}_2]\) yields a quadrilateral \(Q\).
Consider first the case where \(Q\) is convex.
In this case conditions \eqref{quad:dif} and \eqref{quad:x} are by construction fulfilled, so we just need to prove condition \eqref{quad:y}.
As \(Q\) is convex the lines \([\overline{x}_1,\overline{x}_2]\) and \([\overline{y}_2,\overline{y}_2]\) intersect at some point \(\overline{z}\).
By the triangle inequality, the CAT(0) inequality, and as \(\overline{y}_1,\overline{z},\overline{y}_2\) are colinear, we have \(d(y_1,z)+d(z,y_2)\leq d(\overline{y}_1,\overline{z})+d(\overline{z},\overline{y}_2)=d(\overline{y}_1,\overline{y}_2)\).

In the case where \(Q\) is \emph{not} convex, one of the vertices, say \(\overline{x}_2\), is in the interior of the convex hull of \(Q\).
Following Alexandrov's Lemma (\ref{lem:alex}) we create a convex triangle \(\tilde{\Delta}(\tilde{x}_1,\tilde{y}_1,\tilde{y}_2)\) on the edge \(\tilde{x}_2\in [\tilde{y}_1,\tilde{y}_2]\) with \(d(\tilde{x}_i,\tilde{y}_j)=d(x_i,y_j)\).
The law of cosines gives us \(d(\tilde{x}_1,\tilde{x}_2)\geq d(\overline{x}_1,\overline{y}_1)=d(x_1,y_1)\), and the colinearity of \(\tilde{y}_1,\tilde{x}_2\) and \(\tilde{y}_2\) gives us \(d(\tilde{y}_1,\tilde{y}_2)=d(\tilde{y}_1,\tilde{x}_2)+d(\tilde{x}_2,\tilde{y}_2)=d(y_1,x_2)+d(x_2,y_2)\geq d(y_1,y_2)\), showing that \([\tilde{x}_1,\tilde{y}_1,\tilde{x}_2,\tilde{y}_2]\) is a subembedding of \([x_1,y_1,x_2,y_2]\). 
 
\textbf{(2) implies (1):}
First we need to show that geodesics exist.
We do this by showing that midpoints between any two \(x,y\in X\) exist.
As \((X,d)\) is assumed complete we just need to find a Cauchy sequence of approximate midpoints.
Let \(\{m_i\}\in X\) be a sequence of approximate midpoints with \(\max\{d(x,m_i)\}\leq\frac{1}{2}d(x,y)+\frac{1}{i}\), and fix \(\varepsilon>0\).
Let \((\overline{x},\overline{m}_i,\overline{y},\overline{m}_j)\) be a subembedding of \((x,m_i,y,m_j)\).
Then \(d(\overline{m}_i,\overline{m}_j)\geq d(m_i,m_j)\) and \(d(\overline{x},\overline{y})\leq d(\overline{x},\overline{m}_i)+d(\overline{m}_i,\overline{y})=d(x,m_i)+d(m_i,y)\leq d(x,y)+\frac{2}{i}\).
For \(i,j\) sufficiently large we then have \(d(\overline{x},\overline{y})\leq L\) for some constant \(L\).
We know by Proposition \ref{prop:apmid_close} that approximate midpoints are close to midpoints in \(\E^2\), so there exists a \(\delta(L,\varepsilon)\) such that if \(\max\{\frac{1}{i},\frac{1}{j}\}<\delta(\varepsilon,L)\), then \(d(\overline{m}_i,\overline{m})\leq\varepsilon\) and \(d(\overline{m}_j,\overline{m})\leq\varepsilon\), where \(\overline{m}\) is the midpoint of \(\overline{x}\) and \(\overline{y}\) (which exists as \(d(\overline{x},\overline{y})\leq L\)).
Finally, the inequality \(d(\overline{m}_i,\overline{m}_j)\geq d(m_i,m_j)\) shows that \(d(m_i,m_j)<2\varepsilon\) for sufficiently large \(i,j\), whereby \(\{m_i\}\) is Cauchy.
It is therefore convergent, converging to a midpoint of \(x,y\).
\(X\) is assumed complete, so this midpoint is a point in \(X\), which proves that \(X\) admits midpoints.
Using Lemma \ref{lem:midpoints}, this proves the existence of geodesics.

To show the CAT(0) inequality, let \(\Delta=\Delta(x,y,z)\) be a geodesic triangle with a point \(m\in [x,y]\).
Let \((\overline{x},\overline{m},\overline{y},\overline{z})\) be a subembedding in \(\E^2\) of \((x,m,y,z)\).
The equalities in the definition of subembeddings ensure that \(\overline{\Delta}(\overline{x},\overline{y},\overline{z})\) is a comparison triangle of \(\Delta\), where \(\overline{m}\) is the comparison point for \(m\).
Proposition \ref{prop:equiv} (2) and the definition of subembeddings now show that \((X,d)\) is a CAT(0) space.
\end{proof}

%_------------------Convexity-----------------
\section{Convexity of the metric}
One property CAT(0) spaces enjoy is that their metrics are convex functions:
\begin{definition}
A function \(f:X\to\R\) defined on a geodesic metric space \((X,d)\) is \textit{convex} if each geodesic path \(\gamma:[0,1]\to X\) parametrized by arc length satisfies
\[f(\gamma(s))\leq (1-s)f(\gamma(0))+sf(\gamma(1)),\]
for all \(s\in [0,1]\).
\end{definition}
\begin{proposition}\label{prop:conv_met}
If \((X,d)\) is a CAT(0) space, then the metric \(d:X\times X\to\R\) is convex.
\end{proposition}
\begin{proof}
Start out by assuming \(\gamma_1(0)=\gamma_2(0)\), and let \(\overline{\Delta}\subseteq\E^2\) be a comparison triangle for the geodesic triangle \(\Delta(\gamma_1(0),\gamma_1(1),\gamma_2(1))\).
We know from Euclidean geometry that \(d(\overline{\gamma_1(t)},\overline{\gamma_2(t)})=td(\overline{\gamma_1(1)},\overline{\gamma_2(1)})=td(\gamma_1(1),\gamma_2(1))\).
By the CAT(0) inequality we then have \(d(\gamma_1(t),\gamma_2(t))\leq d(\overline{\gamma_1(t)},\overline{\gamma_2(t)})=td(\gamma_1(1),\gamma_2(1))\).
Now for the case where all endpoints are distinct.
Let \(\gamma_3:[0,1]\to X\) be a geodesic such that \(\gamma_3(0)=\gamma_1(0)\) and \(\gamma_3(1)=\gamma_2(1)\).
Applying the special case presented in the beginning of the proof to the geodesics \(\gamma_1\) and \(\gamma_3\), we get \(d(\gamma_1(t),\gamma_3(t))\leq td(\gamma_1(1),\gamma_3(1))\).
Applying it to geodesics \(\gamma_2\) and \(\gamma_3\) (after reversing their orientation), we similarly get \(d(\gamma_2(t),\gamma_3(t))\leq (1-t)d(\gamma_2(0),\gamma_3(0))\).
Combining we get
\[d(\gamma_1(t),\gamma_2(t))\leq d(\gamma_1(t),\gamma_3(t))+d(\gamma_3(t),\gamma_2(t))\leq td(\gamma_1(1),\gamma_3(1))+(1-t)d(\gamma_2(0),\gamma_3(0)),\]
as desired.
\end{proof}

Convexity gives a useful way of identifying flat subspaces in CAT(0) spaces, i.e. subspaces isometric to a subset of Euclidean space. 
Before stating two of the theorems doing this, we will need some preliminaries.
Let \((X,d)\) be a CAT(0) space.
The projection onto a complete geodesically convex subset \(C\subseteq X\) in the sense of Definition \ref{def:geo_con}
is a map \(\pi: X\to C\) mapping \(x\in X\) to the point \(\pi(x)\in X\) such that \(d(x,\pi(x))=d(x,C)=\inf_{y\in C}d(x,y)\).
\begin{lemma}
Let \((X,d)\) be a CAT(0) space with a complete geodesically convex subset \(C\subseteq X\).
Let \(x\in X\).
Then the projection \(\pi(x)\) defined above is unique.
\end{lemma}
\begin{proof}
Assume for contradiction that there are two distinct projections \(\pi(x),\pi(x)'\in C\) of \(x\) onto \(C\).
Note that \(d(x,\pi(x))=d(x,\pi(x)')\), as both are minimal.
Because \(C\) is geodesically convex, the geodesic \([\pi(x),\pi(x)']\) is completely contained within \(C\).
In particular, their midpoint \(m\) is contained in \(C\).
Consider now the comparison triangle \(\overline{\Delta}(\overline{x},\overline{\pi(x)},\overline{\pi(x)'})\).
We have by the CAT(0) inequality
\[d(x,m)\leq d(\overline{x},\overline{m})<d(\overline{x},\overline{\pi(x)}')=d(\overline{x},\overline{\pi(x)})=d(x,\pi(x)).\]
In particular \(d(x,m)<d(x,\pi(x))\), contradicting \(d(x,\pi(x))\) being minimal.
\end{proof}

\begin{proposition}\label{prop:projection_angle}
    Let \(X\) be a CAT(0) space with a complete and geodesically convex subset \(C\).
    Let \(x\notin C\) and \(y\in C\) be points.
    If \(y\neq\pi(x)\), then the Alexandrov angle \(\angle_{\pi(x)}(x,y)\) between \([\pi(x),x]\) and \([\pi(x),y]\) is at least \(\pi/2\).
\end{proposition}
\begin{proof}
First we show that if \(x'\) belongs to the geodesic segment \([x,\pi(x)]\), then \(\pi(x)=\pi(x')\).
As \([x,\pi(x)]\) is a geodesic we have equality in the triangle inequality \(d(x,\pi(x))=d(x,x')+d(x',\pi(x))\).
If \(d(x',C)<d(x',\pi(x))\), i.e. \(\pi(x')\neq \pi(x)\), then \(d(x',\pi(x'))<d(x',\pi(x))\).
The triangle inequality finally gets us
\begin{align*}
d(x,\pi(x'))&\leq d(x,x')+d(x',\pi(x'))\\
            &< d(x,x')+d(x',\pi(x))\\
            &= d(x,\pi(x)),
\end{align*}
implying that \(d(x,\pi(x))>d(x,C)\), a contradiction.
We must therefore have \(\pi(x)=\pi(x')\).

Assume now for contradiction that the Alexandrov angle \(\angle_{\pi(x)}(x,y)<\pi/2\), where \(x\) and \(y\) are as in the proposition.
Then by definition there must be points \(x'\) and \(y'\) on the geodesic segments \([x,\pi(x)]\subseteq X\) and \([y,\pi(x)]\subseteq C\) such that the comparison angle \(\overline{\angle}_{\overline{\pi(x)}}(\overline{x'},\overline{y'})<\pi/2\).
Now choose \(p\in[y',\pi(x)]\) different from \(y'\) and \(\pi(x)\) close enough to \(\pi(x)\) such that the comparison angle \(\overline{\angle}_{\overline{p}}(\overline{x'},\overline{\pi(x)})\) is obtuse.
We now have
\begin{align*}
d(x',\pi(x'))&\leq d(x',p)\tag{as projection}\\
                &\leq d(\overline{x'},\overline{p})\tag{CAT(0)}\\
                &<d(\overline{x'},\overline{\pi(x)})\tag{as obtuse}\\
                &=d(x',\pi(x)),
\end{align*}
where the last equality follows from the fact that \([\overline{x'},\overline{\pi(x)}]\) is an edge in the comparison triangle.
But earlier we showed that \(\pi(x')=\pi(x)\) so \(d(x',\pi(x'))<d(x',\pi(x'))\), a contradiction.
The assumption \(\angle_{\pi(x)}(x,y)<\pi/2\) is therefore wrong, so \(\angle_{\pi(x)}(x,y)\geq \pi/2\) as desired.
\end{proof}

\subsection{Flat polygons}
We define the \textit{convex hull} of a subset \(A\) of a geodesic space \(X\) to be the intersection of all closed geodesically convex subsets of \(X\) containing \(A\), i.e. as the ``smallest'' convex subset containing \(A\).
The following proposition states that if an angle in a geodesic triangle is equal to its corresponding comparison angle, then its convex hull is flat.
\begin{lemma}[Flat Triangle Lemma]\label{lem:flat_tri}
Let \(\Delta\) be a geodesic triangle in a CAT(0) space \(X\).
If one of the vertex angles of \(\Delta\) is equal to the corresponding comparison angle in a comparison triangle \(\overline{\Delta}\subseteq\E^2\) of \(\Delta\), then the convex hull of \(\Delta\) in \(X\) is isometric to the convex hull of \(\overline{\Delta}\) in \(\E^2\).
\end{lemma}
\begin{proof}
Let \(\Delta=\Delta(p,q,q')\) be a geodesic triangle in \(X\) with comparison triangle \(\overline{\Delta}=\overline{\Delta}(\overline{p},\overline{q},\overline{q'})\), and assume \(\angle_p(q,q')=\overline{\angle}_p(q,q')\).
Let \(r\in [q,q']\) be a point.
Next let \(\Delta'=\Delta(p,q,r)\) and \(\Delta''(p,q',r)\) be geodesic triangles with comparison triangles \(\twiddle{\Delta}'=\overline{\Delta}(\twiddle{p},\twiddle{q},\twiddle{r})\) and \(\twiddle{\Delta}''=\overline{\Delta}(\twiddle{p},\twiddle{r},\twiddle{q}')\).
Arrange these in \(\E^2\) such that they share the edge \([\twiddle{p},\twiddle{r}]\) with \(\twiddle{q}\) and \(\twiddle{q}'\) lie on opposite sides of the line \([\twiddle{p},\twiddle{r}]\).
We have
\begin{align*}
\angle_p(q,q')  &\leq \angle_p(q,r)+\angle_p(r,q')\tag{Proposition \ref{prop:pseudo}}\\
&\leq \angle_{\twiddle{p}}(\twiddle{q},\twiddle{r})+\angle_{\twiddle{p}} (\twiddle{r},\twiddle{q}')\tag{Proposition \ref{prop:equiv}}\\
&\leq \angle_{\overline{p}}(\overline{q},\overline{q}')
\end{align*}
The last inequality is obtained as follows: The line \([\twiddle{q},\twiddle{q'}]\) has length shorter than \(d(\twiddle{q},\twiddle{r})+d(\twiddle{q'},\twiddle{r})=d(\overline{q},\overline{q'})\) by the triangle inequality.
The law of cosines and the CAT(0) condition on \(X\) then imply the inequality.
As we assumed \(\angle_p(q,q')=\overline{\angle}_p(q,q')\) we must have equality everywhere, so \(\angle_{\twiddle{p}}(\twiddle{q},\twiddle{r})+\angle_{\twiddle{p}} (\twiddle{r},\twiddle{q}')=\angle_{\overline{p}}(\overline{q},\overline{q}')\).
This furthermore implies \(d(p,r)=d(\twiddle{p},\twiddle{r})=d(\overline{p},\overline{r})\).

Next let \(j:C(\overline{\Delta})\to X\) be the map which maps the geodesic segment \([\overline{p},\overline{r}]\) isometrically onto the geodesic segment \([p,r]\), where \(\overline{r}\in [\overline{q},\overline{q'}]\), and where \(C(\overline{\Delta})\) denotes the convex hull of \(\overline{\Delta}\).
Let \(\overline{r}\) and \(\overline{r}'\) be points on \([\overline{q},\overline{q}']\) such that \(\overline{r}\in[\overline{q},\overline{r}']\), and let \(\overline{x}\in [\overline{p},\overline{r}],\overline{x}'\in[\overline{p},\overline{r}']\).
Next let \(x=j(\overline{x}),x'=j(\overline{x}'),r=j(\overline{r})\) and \(r'=j(\overline{r}')\) be points, and let \(\delta_1=\angle_p(q,r),\delta_2=\angle_p(r,r'),\delta_3=\angle_p(r',q)\) be Alexandrov angles.
As shown above \(\Delta(\overline{p},\overline{q},\overline{r})\) is a comparison triangle for \(\Delta(p,q,r)\), so \(\delta_1\leq\overline{\delta}_1\), where \(\overline{\delta}_1\) is the angle in \(\overline{\Delta}\) corresponding to \(\delta_1\).
In the same way we get \(\delta_2\leq\overline{\delta}_2\) and \(\delta_3\leq\overline{\delta}_3\).
As \(\angle\) is a pseudometric we have \(\angle_p(q,q')\leq\delta_1+\delta_2+\delta_3\leq\overline{\delta}_1+\overline{\delta}_2+\overline{\delta}_3=\angle_{\overline{p}}(\overline{q},\overline{q}')\). 
By assumption this equals \(\angle_p(q,q')\), so we have equality everywhere, whereby \(\delta_2=\overline{\delta}_2\).
This finally gets us \(d(x,x')=d(\overline{x},\overline{x}')\), showing that \(j\) is an isometry onto the convex hull \(C(\Delta)\).
\end{proof}

We have now shown that a geodesic triangle having one angle the same as its comparison angle in the comparison triangle implies that it is ``flat''.
The next theorem states something similar, giving us a condition for geodesic quadrilaterals to be flat:
\begin{theorem}[The Flat Quadrilateral Theorem]\label{thm:flat_quad}
    Let \(p,q,r,s\) be points in a CAT(0) space \((X,d)\), and let \(\alpha=\overline{\angle}_p(s,q),\beta=\overline{\angle}_q(p,r),\gamma=\overline{\angle}_r(q,s)\) and \(\delta=\overline{\angle}_s(r,p)\) be comparison angles.
    If \(\alpha+\beta+\gamma+\delta\geq 2\pi\) then \(\alpha+\beta+\gamma+\delta= 2\pi\), and the convex hull of the four points is isometric to a convex quadrilateral in \(\E^2\).    
\end{theorem}
\begin{proof}
Let \(\Delta_1=\Delta(p,q,s),\Delta_2=\Delta(r,q,s)\), and arrange their comparison triangles in \(\E^2\) such that they share the line segment \([\overline{q},\overline{s}]\).
Let \(\overline{\Delta}_1\) have angles \(\overline{\alpha},\overline{\beta}_1,\overline{\delta}_1\) at \(\overline{p},\overline{q}\) and \(\overline{s}\), respectively, and let \(\overline{\Delta}_2\) have angles \(\overline{\gamma},\overline{\beta}_2,\overline{\delta}_2\) at \(\overline{r},\overline{q}\) and \(\overline{s}\), respectively.
Let \(\alpha,\beta_1,\beta_2,\gamma,\delta_1,\delta_2\) be the corresponding Alexandrov angles in the corresponding geodesic triangles in \(X\).
As \(X\) is a CAT(0) space, proposition \ref{prop:equiv} tells us that \(\alpha\leq\overline{\alpha}\) and \(\gamma\leq\overline{\gamma}\).
By the triangle inequality for angles (cf. Proposition \ref{prop:pseudo}) we also have \(\beta\leq\beta_1+\beta_2\leq\overline{\beta}_1+\overline{\beta}_2\) and \(\delta\leq \delta_1+\delta_2\leq\overline{\delta}_1+\overline{\delta}_2\).
Under the assumption \(\alpha+\beta+\gamma+\delta\geq2\pi\) we then have \(\overline{\alpha}+\overline{\beta}_1+\overline{\beta}_2+\overline{\gamma}+\overline{\delta}_1+\overline{\delta}_2\geq 2\pi\).
But this sum is the angle sum of a quadrilateral in \(\E^2\), so we must have equality everywhere.
We must in particular have \(\overline{\beta}=\overline{\beta}_1+\overline{\beta}_2\leq\pi\) and \(\overline{\delta}=\overline{\delta}_1+\overline{\delta}_2\leq\pi\).
This means that the quadrilateral \(Q\) in \(\E^2\) with vertices \(\overline{p},\overline{q},\overline{r}\) and \(\overline{s}\) must be convex.
Let \(j:C(Q)\to X\) be a map such that the restrictions \(j_1:C(\overline{\Delta}_1)\to\Delta_1\) and \(j_2:C(\overline{\Delta}_2)\to\Delta_2\) are isometries (these exist by the Flat Triangle Lemma).
We know that distance between two points \(\overline{x}_1\) and \(\overline{x}_2\) is preserved if both points lie in the same comparison triangle, so all that is left to show is that distance is preserved when one point lies in \(C(\overline{\Delta}_1)\) and the other lies in \(C(\overline{\Delta}_2)\).
Let \(\overline{x}_1\in C(\overline{\Delta_1}),\overline{x}_2\in C(\overline{\Delta_2})\) be two such points, and let \(\overline{\mu}=\angle_{\overline{p}}(\overline{x}_1,\overline{x}_2)\).
Denote the corresponding angle in \(X\) by \(\mu\).
Let \(\beta'_1=\angle_q(p,x_1),\beta''_1=\angle_q(x_1,s),\beta''_2=\angle_q(s,x_2)\) and \(\beta_2'=\angle_q(x_2,r)\).
We now have
\begin{align*}
\beta   &\leq\beta'_1+\mu+\beta_2'\tag{Proposition \ref{prop:pseudo}}\\
        &\leq\beta'_1+\beta''_1+\beta''_2+\beta'_2\tag{Proposition \ref{prop:pseudo}}\\
        &\leq\overline{\beta}'_1+\overline{\beta}''_1+\overline{\beta}''_2+\overline{\beta}'_2\\
        &=\overline{\beta}_1+\overline{\beta}_2\\
        &=\beta,
\end{align*}
giving us equalities everywhere. 
This gives us \(\mu=\overline{\beta}''_1+\overline{\beta}''_2\), whereby \(\mu=\overline{\mu}\).
This together with the Flat Triangle Lemma \ref{lem:flat_tri} finally gives us the desired \(d(x_1,x_2)=d(\overline{x}_1,\overline{x}_2)\).
\end{proof}

\begin{theorem}[The Flat Strip Theorem]\label{thm:flat_strip}
    Let \((X,d)\) be a CAT(0) space with two geodesic lines \(\gamma_1:\R\to X,\gamma_2:\R\to X\).
    If \(\gamma_1\) and \(\gamma_2\) are asymptotic, that is, if there is a constant \(K\) such that \(d(c(t),c'(t))\leq K\) for \(t\in\R\), then the convex hull of \(\gamma_1(\R)\cup\gamma_2(\R)\) in \(X\) is isometric to a flat strip \(\R\times [0,D]\subseteq\E^2\) for some \(0<D\in\R\).
\end{theorem}
\begin{proof}
Let \(\pi\) be the projection of \(X\) onto the complete convex subset \(\gamma_1(\R)\).
Reparametrizing allows us to assume that \(\pi(\gamma_2(0))=\gamma_1(0)\), whereby \(\gamma_1(0)\) is the point on \(\gamma_1(\R)\) closest to \(\gamma_2(0)\).
By Proposition \ref{prop:conv_met} we know that \(t\mapsto d(\gamma_1(t),\gamma_2(t))\) is a function that is positive and convex.
It is also bounded as \(\gamma_1\) and \(\gamma_2\) are asymptotic.
These three conditions togther imply that the function \(t\mapsto d(\gamma_1(t),\gamma_2(t))\) constant, equal to some constant \(D\), as bounded convex functions are constant.
\(t\mapsto d(\gamma_1(t+a),\gamma_2(t))\) is also constant, so \(d(\gamma_1(t+a),\gamma_2(t))=d(\gamma_1(a),\gamma_2(0))\geq d(\gamma_1(0),\gamma_2(0))\).
This implies that \(d(\gamma_1(t+a),\gamma_2(t))\) is smallest when \(a=0\), so \(\pi(\gamma_2(t))=\gamma_1(t)\).
Similarly the projection \(\pi':X\to \gamma_2(\R)\) satisfies \(\pi'(\gamma_1(t))=\gamma_2(t)\).

We now let \(t<t'\) be real numbers, and consider the quadrilateral \([\gamma_1(t),\gamma_1(t'),\gamma_2(t'),\gamma_1(t)]\).
By Proposition \ref{prop:projection_angle}, all angles are at least \(\pi/2\), so their sum is no less than \(2\pi\).
The Flat Quadrilateral Theorem \ref{thm:flat_quad} now ensures that the sum is exactly \(2\pi\), so all angles must be exactly \(2\pi\).
The Flat Quadrilateral Theorem further states that the convex hull of the points \(\gamma_1(t),\gamma_1(t'),\gamma_2(t')\) and \(\gamma_2(t)\) is isometric to a rectangle in \(\E^2\), so the map \(j:\R\times[0,D]\to X\) mapping \((t,s)\in \R\times[0,D]\) to the point on \([\gamma_1(t),\gamma_2(t)]\) a distance \(s\) away from \(\gamma_1(t)\) is an isometry onto the convex hull of \(\gamma_1(\R)\cup\gamma_2(\R)\).
\end{proof}
This can be paraphrased as stating that asymptotic geodesic lines in a CAT(0) space are parallel.

\begin{theorem}[A Product Decomposition Theorem]
    Let \((X,d)\) be a CAT(0) space, let \(\gamma:\R\to X\) be a geodesic line, and denote by \(\pi:X\to \gamma(\R)\) the projection from \(X\) to the geodesically convex subspace \(\gamma(\R)\).
\parenum
\begin{enumerate}
    \item The union of the images of all geodesic lines \(\gamma_p\) in \(X\) asymptotic to \(\gamma\) is a geodesically convex subspace \(X_\gamma\) of \(X\), and \(\left(X_\gamma,d|_{X_\gamma}\right)\) is a CAT(0) space;
    \item \(X_\gamma\) is canonically isometric to \(X_\gamma^0\times \R\), where \(X_\gamma^0=\pi_\gamma^{-1}(\gamma(0))\), and \(X_\gamma^0\) is geodesically convex in \(X\).
\end{enumerate}
\end{theorem}
\begin{proof}
\parenum
\begin{enumerate}
\item Let \(x_1,x_2\in X_\gamma\) be points, and let \(\gamma_1,\gamma_2\) be geodesic lines parallel to \(\gamma\) containing \(x_1\) and \(x_2\), respectively.
In particular \(\gamma_1\) and \(\gamma_2\) are asymptotic, so the Flat Strip Theorem \ref{thm:flat_strip} tells us that the convex hull in \(X\) of the geodesic lines \(\gamma_1,\gamma_2\) is isometric to a flat strip \(\R\times [0,D]\) for some real number \(0<D<\infty\).
Then the geodesic in \(X\) between \(x_1\) and \(x_2\) is contained in the convex hull of \(\gamma_1(\R)\cup\gamma_2(\R)\) in \(X\), which is contained in \(X_\gamma\).
\(X_\gamma\) is therefore a geodesically convex subspace of the CAT(0) space \(X\), and thus itself a CAT(0) space.

\item First we show that for all \(x\in X_\gamma\) there is a unique geodesic \(\gamma_x\) parallel to \(\gamma\) containing the point \(x\).
Assume there are geodesics \(\gamma_1\) and \(\gamma_2\) both parallel to \(\gamma\) and both containing \(x\).
As \(X_\gamma\) is a CAT(0) space, the function \(t\mapsto d(\gamma_1(t),\gamma_2(t))\) is convex by Proposition \ref{prop:conv_met}.
Both \(\gamma_1\) and \(\gamma_2\) are parallel to \(\gamma\), and therefore asymptotic to \(\gamma\).
They must therefore also be asymptotic to each other, i.e. \(t\mapsto d(\gamma_1(t),\gamma_2(t))\) is bounded.
As it is also convex it must be constant.
We now reparametrize \(\gamma_1\) and \(\gamma_2\) such that \(\gamma_1(0)=\gamma_2(0)=x\).
In particular \(d(\gamma_1(0),\gamma_2(0))\), and as this is constant we must have \(d(\gamma_1(t),\gamma_2(t))=0\) for all \(t\in\R\).
This means that \(\gamma_1\) and \(\gamma_2\) are the same geodesic, up to reparametrization.

Next, for any two geodesics \(\gamma_i,\gamma_j\) we define functions \(P_{\gamma_i,\gamma_j}(t)\) to be the point on \(\gamma_j\) closest to \(\gamma_i(t)\).
We claim that \(P=P_{\gamma_1,\gamma_3}\circ P_{\gamma_3,\gamma_2}\circ P_{\gamma_2, \gamma_1}=P_{\gamma_1,\gamma_1}\) is the identity on \(\gamma_1\).
Assume not.
Then \(P(t)=\gamma_1(t+b)\) for all \(t\in\R\) and some real non-zero \(b\in\R\).
Next define \(a_1=d(\gamma_1(\R),\gamma_2(\R)),a_2=d(\gamma_2(\R),\gamma_3(\R))\) and \(a_3=d(\gamma_3(\R),\gamma_1(\R))\), and let \(a=a_1+a_2+a_3\).
All pairs of geodesics are parallel, the Flat Strip Theorem tells us that they are parallel in \(\E^2\).
Here we have
\begin{align*}
d(\gamma_1(t),\gamma_2(t+s))    &=\sqrt{a_1^2+s^2},\\
d(\gamma_2(t),\gamma_3(t+s))    &=\sqrt{a_2^2+s^2},\\
d(\gamma_3(t),\gamma_1(t+s)     &=\sqrt{a_3^2+(s-b)^2},
\end{align*}
for all \(s\in\R\).
This in turn gives us 
\begin{align*}
d(\gamma_1(0),\gamma_1(as+b)&\leq d(\gamma_1(0),\gamma_2(a_1s))+d(\gamma_2(a_1s),\gamma_3((a_1+a_2)s))+d(\gamma_3((a_1+a_2)s),\gamma_1(as+b))\\
&= a_1\sqrt{1+s^2}+a_2\sqrt{1+s^2}+a_3\sqrt{1+s^s}\\
&= a\sqrt{s^2}.
\end{align*}
As \(\gamma_1\) is a geodesic we have \(d(\gamma_1(0),\gamma_1(as+b))=|as+b|\), so \((as+b)^2\leq a^2(1+s^2)\).
This implies \(b=0\), a contradiction.

We are now ready to show that \(X_\gamma^0=\pi_\gamma^{-1}(\gamma(0))\) is geodesically convex in \(X\), i.e. that for every \(x,y\in X_\gamma^0\), the geodesic segment \([x,y]\) is entirely contained in \(X_\gamma^0\).
Let \(\gamma_x,\gamma_y\) be the unique geodesics parallel to \(\gamma\) containing \(x\) and \(y\), respectively, reparametrized such that \(\gamma_x(0)=x\) and \(\gamma_y(0)=y\).
By the Flat Strip Theorem and part (1) of the proposition, the geodesic \([x,y]\) is entirely contained in \(X_\gamma\).
As was shown before,
\[P_{\gamma,\gamma_x}\circ P_{\gamma_x,\gamma_y}\circ P_{\gamma_y,\gamma}=P_{\gamma,\gamma}.\]
As \(d(\gamma,\gamma_x)=d(\gamma(0),\gamma_x(0))=d(\gamma(0),x)\), we get
\[\gamma(0)=P_{\gamma,\gamma}(0)=P_{\gamma,\gamma_x}\circ P_{\gamma_x,\gamma_y}\circ P_{\gamma_y,\gamma}(0)=P_{\gamma,\gamma_y}\circ P_{\gamma_y,\gamma_x}(\gamma_x(0)),\]
which is true if and only if
\[P_{\gamma,\gamma_y}^{-1}(\gamma(0))=P_{\gamma_y,\gamma_x}(\gamma_0)=P_{\gamma_y,\gamma_x}(x)\]
The left hand side is simply \(\gamma_y(0)=y\), so
\[y=P_{\gamma_y,\gamma_x}(x).\]
This finally gives us
\[d(\gamma_x,\gamma_y)=d(\gamma_x(0),\gamma_y(0))=d(x,y)\]
for all \(x,y\in X_\gamma^0\).
The Flat Strip Theorem then lets us conclude that \([x,y]\) is entirely contained in \(X_\gamma^0\) for all \(x,y\in X_{\gamma}^0\), proving that \(X_\gamma^0\) is indeed geodesically convex in \(X\).

To show that \(X_\gamma\) is isometric to \(X_\gamma^0\times\R\) we define the map \(j:X_\gamma\times\R\to X_\gamma\) by \((x,t)\mapsto \gamma_x(t)\).
Each \(x\in X_\gamma\) is contained in some geodesic \(\gamma_x\) asymptotic to \(\gamma\), so \(j\) is surjective.
Before we showed that such a geodesic is unique, so \(j\) is also injective, and thereby a bijection.
Endowing \(X_\gamma^0\) with the product metric between \(d|_{X_\gamma^0}\) and \(d_\R\)
\[d((x,t),(y,t'))=\sqrt{d(x,y)^2+|t-t'|^2}\]
for all \((x,t),(y,t')\in X_{\gamma}^0\), \(j\) being an isometry follows directly from the Flat Strip Theorem.
\end{enumerate}
\end{proof}

%---------------CHT---------------------
\section{The Cartan--Hadamard Theorem}
Notice that all different notions of non-positive curvature discussed above are \textit{local} conditions.
The Cartan--Hadamard Theorem tells us a global consequence of non-positive curvature:
\begin{theorem}[The Cartan--Hadamard Theorem]\label{thm:CHT}
    Let \((X,d)\) be a complete length space that is  path connected (i.e. between every two points there is a continuous path).
    Then:
\parenum\begin{enumerate}
    \item There is a unique length metric \(\twiddle{d}\) on the universal cover \(\twiddle{X}\) such that the covering map \(p:(\twiddle{X},\twiddle{d})\to (X,d)\) is a local isometry;
    \item If the metric \(d\) is locally convex, i.e. if every \(x\in X\) has a geodesically convex neighborhood where the induced metric is convex, then \(\twiddle{d}\) is globally convex;
    \item \(\twiddle{X}\) is uniquely geodesic;
    \item If \(X\) is a locally CAT(0) space (see Definition \ref{def:loc_cat}), then \((\twiddle{X},\twiddle{d})\) is a (globally) CAT(0) space.
\end{enumerate}
\end{theorem}

\subsection{The induced length metric on the universal cover}
For the first part of Theorem \ref{thm:CHT} we will need to first construct a length metric \(\twiddle{d}\) on the universal cover \(\twiddle{X}\) of a complete connected metric space.
\begin{definition}
\label{def::ind_metric_univ}
Let \((X,d)\) be a length space with universal cover $\twiddle{X}$, and let \(p\) be the covering map \(p:\twiddle{X}\to X\). 
Given a path \(\sigma:[0,1]\to\twiddle{X}\) we define its length to be the length of \(p\circ\sigma\) in \(X\).
We then define a metric \(\twiddle{d}\) on \(\twiddle{X}\) by letting \(\twiddle{d}(\twiddle{x},\twiddle{y})\) be the infimum of the lengths of all rectifiable curves between \(\twiddle{x}\in\twiddle{X}\) and \(\twiddle{y}\in\twiddle{X}\), completely analogously to the length metric on metric spaces.
We call \(\twiddle{d}\) the \textit{metric induced on \(\twiddle{X}\) by \(p\)}.
\end{definition}
\begin{remark}
This is indeed a metric: the length metric in \(X\) is a metric, so \(\twiddle{d}\) must be non-negative, symmetric, and satisfy the triangle inequality.
We also know that \(\twiddle{X}\) is Hausdorff as it is the covering space of a metric space, all of which are Hausdorff. Therefore \(\twiddle{d}\) must also necessarily be positively definite. 
\end{remark}

\begin{proposition}\label{prop:p_loc_iso}
Let \(X\) be a length space with universal cover \(\twiddle{X}\) equipped with the metric \(\twiddle{d}\) induced on \(\twiddle{X}\) by the covering map \(p\).
Then:\parenum
\begin{enumerate}
    \item \(p\) is a local isometry;
    \item \(\twiddle{d}\) is a length metric;
    \item \(\twiddle{d}\) is the only metric on \(\twiddle{X}\) satisfying (1) and (2).
\end{enumerate}
\end{proposition}
\begin{proof}
\parenum\begin{enumerate}
    \item Let \(\twiddle{x}\in\twiddle{X}\) be given, and let \(x=p(\twiddle{x})\).
We show that the restriction of \(p\) to \(B(\twiddle{x},r)\) is an isometry onto \(B(x,r)\) for sufficiently small \(r\).
Let \(\twiddle{U}_{\twiddle{x}}\) be an open neighborhood of \(\twiddle{x}\) such that \(p|_{\twiddle{U}_{\twiddle{x}}}\) is a homeomorphism onto its image \(U_x\), which is possible as covering maps by definition are local homeomorphisms.
Denote the inverse \(s=p^{-1}:U_x\to\twiddle{U_{\twiddle{x}}}\), which is then continuous and bijective.
Choose \(r>0\) small enough such that \(B(x,3r)\subseteq U_x\) and \(B(\twiddle{x},3r)\subseteq\twiddle{U}_{x}\), where $B(\twiddle{x},3r):= \{ \twiddle{y} \in \twiddle{U}_{x}  \; \vert \; \twiddle{d}(\twiddle{x},\twiddle{y}) < 3r\}$.
As \(s\) is continuous and injective, so is \(s|_{B(x,r)}\).
This in turn means that \(s|_{B(x,r)}\) is a homeomorphism onto its image \(Y_x\subseteq\twiddle{U}_{\twiddle{x}}\).
\(B(x,r)\) is open, so we conclude that \(Y_x\) is open in \(\twiddle{X}\).

First we show that \(s|_{B(x,r)}\) preserves distance, i.e. is an isometry between \(B(x,r)\) and \(Y_x\).
Let \(\twiddle{y},\twiddle{z}\in Y_x\), and take \(0<\varepsilon<r\).
Define \(y=p(\twiddle{y})\in B(x,r)\) and \(z=p(\twiddle{z})\in B(x,r)\).
\(X\) is a length space, so there is a curve \(\sigma:[0,1]\to X\) between \(y\) and \(z\) whose length is strictly smaller than \(d(y,z)+\varepsilon\).
Since both are in the open ball \(B(x,r)\) we also have \(d(y,z)<2r\), so we can choose a very small \(\varepsilon\) such that \(d(y,z)+\varepsilon<2r\), whereby \(L(\sigma)<2r\).
We conclude that \(\sigma\) is entirely contained in the open ball \(B(x,3r)\subseteq U_x\) as
\[d(x,\sigma(t))\leq d(x,y)+d(y,\sigma(t))<2r+r=3r.\]
This in turn means that \(s\circ\sigma=\twiddle{\sigma}\) is a curve entirely contained in \(\twiddle{U}_{\twiddle{x}}\) between \(\twiddle{y}\) and \(\twiddle{z}\) with \(p\circ\twiddle{\sigma}=\sigma\), whereby 
\[d(y,z)\leq \twiddle{d}(\twiddle{y},\twiddle{z})\leq L(p\circ\twiddle{\sigma})=L(\sigma)<d(y,z)+\varepsilon.\]
As \(\varepsilon\) can be chosen arbitrarily small we get \(d(y,z)=\twiddle{d}(\twiddle{y},\twiddle{z})\), showing that \(s\) preserves distances.

In particular we have, \(\twiddle{d}(\twiddle{x},\twiddle{y})=d(x,y)<r\), so \(Y_x\subseteq B(\twiddle{x},r) := \{ \twiddle{y} \in \twiddle{U}_{x}  \; \vert \; \twiddle{d}(\twiddle{x},\twiddle{y}) < r\} \).
Left is to show that \(B(\twiddle{x},r)\subseteq Y_x\).
By the choice of \(r\) we have \(B(\twiddle{x},r)\subseteq B(\twiddle{x},3r)\subseteq \twiddle{U}_{\twiddle{x}}\), and as \(p:\twiddle{U}_{\twiddle{x}}\to U_x\) is a bijection we have \(s(z)=s(p(\twiddle{z}))=\twiddle{z}\in B(\twiddle{x},r)\) for any \(\twiddle{z}\in B(\twiddle{x},r)\) with \(z=p(\twiddle{z})\).
But we know \(z=p(\twiddle{z})\in B(x,r)\) as \(d(x,z)\leq\twiddle{d}(\twiddle{x},\twiddle{z})<r\), and as \(s|_{B(x,r)}:B(x,r)\to Y_x\) is a bijection we must have \(s(z)=\twiddle{z}\in Y_x\).
This in turn means that \(B(\twiddle{x},r)\subseteq Y_x\), whereby \(Y_x=B(\twiddle{x},r)\).

We have now shown that every \(\twiddle{x}\in \twiddle{X}\) has a neighborhood which is isometric to a neighborhood of \(x=p(\twiddle{x})\), proving (1).

    \item By (1) \(p\) is a local isometry, so the length of a curve \(\sigma\) in \(\twiddle{X}\) is the same as its image \(p(\sigma)\) in \(X\), whereby \(\twiddle{d}\) is necessarily a length metric on \(\twiddle{X}\).

    \item Let \(\twiddle{d}'\) be a metric satisfying (1) and (2).
Then the identity map \(\text{id}:(\twiddle{X},\twiddle{d})\to (\twiddle{X},\twiddle{d}')\) is a local isometry, so the length of \textit{all} curves is preserved.
This means that \(\twiddle{d}(\twiddle{x},\twiddle{y})=\twiddle{d}'(\twiddle{x},\twiddle{y})\) for all \(\twiddle{x},\twiddle{y}\in\twiddle{X}\), so \(\twiddle{d}'\) is necessarily \(\twiddle{d}\).
\end{enumerate}
\end{proof}
\begin{remark}
The proof doesn't require \(p\) to be a covering map, it is enough that it is a local homeomorphism.
We can therefore replace ``universal cover'' and ``covering map'' with ``Hausdorff topological space'' and ``local homeomorphism'', and the theorem still holds (we need the topological space to be Hausdorff to ensure that \(\twiddle{d}\) is a metric).
We are however only interested in the universal cover of a metric space which is always Hausdorff, so the theorem has been stated in terms of this.
\end{remark}

Proposition \ref{prop:p_loc_iso} proves part (1) of Theorem \ref{thm:CHT}.

\subsection{Convexity of the induced length metric}
\begin{lemma}\label{lem:induced_convex}
Let \((X,d)\) be a metric space with locally complete and locally convex metric \(d\), i.e. where each point \(x\in X\) is contained in a geodesically convex and complete neighborhood \(U_x\) in which the induced metric is convex (see Proposition \ref{prop:conv_mid}).
Let \(\sigma:[0,1]\to X\) be a local geodesic between two points \(x,y\in X\), and let \(\varepsilon>0\) be such that for every \(t\in[0,1]\), the induced metric on the closed ball \(\overline{B}(\sigma(t),2\varepsilon)\) is complete and convex.
Then:
\parenum\begin{enumerate}
    \item For every points \(\overline{x},\overline{y}\in X\) with the properties \(d(x,\overline{x})<\varepsilon\) and \(d(y,\overline{y})<\varepsilon\), there is a unique local geodesic \(\overline{\sigma}:[0,1]\to X\) connecting \(\overline{x}\) to \(\overline{y}\) such that the map \(t\mapsto d(\sigma(t),\overline{\sigma}(t))\) is convex;
    \item \[L(\overline{\sigma})\leq L(\sigma)+d(x,\overline{x})+d(y,\overline{y}).\]
\end{enumerate}
\end{lemma}
\begin{proof}
We start by assuming that such a \(\overline{\sigma}\) exists.
First we show uniqueness.
Let \(\sigma_1,\sigma_2:[0,1]\to X\) be local geodesics such that condition (1) holds for both $\sigma_1,\sigma_2$. Then we have that for any such two local geodesics \(\sigma_1,\sigma_2\) satisfying the conditions of part (1), the map \(t\mapsto d(\sigma_1(t),\sigma_2(t))\) is convex. Indeed, first notice that 
$$d(\sigma_1(t),\sigma(t))\leq (1-t)d(\sigma_1(0),\sigma(0))+td(\sigma_1(1),\sigma(1))\leq (1-t)\varepsilon+\varepsilon=\varepsilon,$$ 
for all \(t\in [0,1]\) (and similarly for \(\sigma_2\)).
For all \(t\in [0,1]\) we have \(\sigma_1(t),\sigma_2(t)\in B(\sigma(t),2\varepsilon)\), which is a convex neighborhood of \(\sigma(t)\).
For every \(t\in [0,1]\) there is then a \(\delta_t>0\) such that \(v\mapsto d(\sigma_1(v),\sigma_2(v))\) is convex on \([t-\delta_t,t+\delta_t]\).
In particular, \(t\mapsto d(\sigma_1(t),\sigma_2(t))\) is locally convex, and therefore convex by a general theorem.
Now, by convexity and since \(\sigma_1(0)=\sigma_2(0)=\overline{x}\) and \(\sigma_1(1)=\sigma_2(1)=\overline{y}\) we must then have \(0\leq d(\sigma_1(t),\sigma_2(t))\leq (1-t)d(\sigma_1(0),\sigma_2(0))+td(\sigma_1(1),\sigma_2(1))=0\), so \(\sigma_1=\sigma_2\).
This shows that if a local geodesic satisfying the conditions in part (1) exists, it must necessarily be unique.

\medskip
Next we show that this local geodesic $\overline{\sigma}$ from (1) will satisfy the inequality in part (2).

 By the first part of the proof it is enough to consider local geodesics \(\sigma_1,\sigma_2:[0,1]\to X\) with \(d(\sigma(t),\sigma_1(t))<\varepsilon\) and \(d(\sigma(t),\sigma_2(t))<\varepsilon\) for all \(t\in [0,1]\), but this time we also assume \(\sigma_1(0)=\sigma_2(0)\). Notice 
\(\sigma_1(t)\) and \(\sigma_2(t)\) are both in the convex neighborhood \(B(\sigma(t),2\varepsilon)\) when \(t\in [0,1]\). So again by the convexity proven above \(d(\sigma_1(t),\sigma_2(t))\leq (1-t)d(\sigma_1(0),\sigma_2(0))+td(\sigma_1(1),\sigma_2(1))=td(\sigma_1(1),\sigma_2(1))\).
As \(\sigma_1\) and \(\sigma_2\) are local geodesics, \(\sigma_1|_{[0,t]}\) and \(\sigma_2|_{[0,t]}\) are geodesics for sufficiently small \(t\in (0,1]\).
In particular, \(L(\sigma_1|_{[0,t]})=tL(\sigma_1)\) and \(L(\sigma_2|_{[0,t]})=d(\sigma_2(0),\sigma_2(t))\).
We now have
\begin{align*}
tL(\sigma_2)&=d(\sigma_2(0),\sigma_2(t))=d(\sigma_1(0),\sigma_2(t))\\
            &\leq d(\sigma_1(0),\sigma_1(t))+d(\sigma_1(t),\sigma_2(t))\\
            &\leq tL(\sigma_1)+td(\sigma_1(1),\sigma_2(1)),
\end{align*}
for sufficiently small \(t\in (0,1]\).
Dividing through by \(t\) gives us
\begin{equation}\label{eq:length_ineq}L(\sigma_2)\leq L(\sigma_1)+d(\sigma_1(1),\sigma_2(1)),\end{equation}
for local geodesics \(\sigma_1,\sigma_2\) with \(\sigma_1(0)=\sigma_2(0)\).

Now let \(x,y,\overline{x},\overline{y}\) and \(\overline{\sigma}\) be as in part (1) of this Lemma. Let also \(\overline{\sigma_1}:[0,1]\to X\) be a local geodesic satisfying the condition (1) of this Lemma and with the property that $\overline{\sigma_1}(0)=\overline{\sigma}(0)=\overline{x}$ and $\overline{\sigma_1}(1)=\sigma(1)=y$.

Above we showed that those local geodesics $\overline{\sigma}, \overline{\sigma_1}$ are unique if they exist. Existence will be shown later. Then we apply inequality \eqref{eq:length_ineq} to get 
\begin{equation}\label{eq:length_ineq1}
L(\overline{\sigma})\leq L(\overline{\sigma_1})+d(\overline{\sigma}(1)),\overline{\sigma_1}(1)=L(\overline{\sigma_1})+d(y,\overline{y}).
\end{equation}
We also have \(-\overline{\sigma_1}(0)=-\sigma(0)=y\), so we can again apply inequality \eqref{eq:length_ineq}, which gives us \begin{align}
L(-\overline{\sigma_1})&\leq L(-\sigma)+d(-\overline{\sigma_1}(1),-\sigma(1)) \notag\\
L(\overline{\sigma_1})&\leq L(\sigma)+d(\overline{x},x)\label{eq:length_ineq2}.
\end{align}
Inserting inequality \eqref{eq:length_ineq1} into inequality \eqref{eq:length_ineq2} finally gives us
\[L(\overline{\sigma})\leq L(\sigma)+d(x,\overline{x})+d(y,\overline{y})\]
as desired.

\medskip
All of the above assumes the existence of such a \(\overline{\sigma}\).
We will now show \(\overline{\sigma}\) always exists.
Let \(A>0\), and consider the following statement:
\begin{enumerate}
    \item[\(P(A)\)]
For all \(a,b\in [0,1]\) with \(0<b-a\leq A\) and for all \(\overline{p},\overline{q}\in X\) with \(d(\sigma(a),\overline{p})<\varepsilon\) and \(d(\sigma(b),\overline{q})<\varepsilon\) there exists a local geodesic \(\overline{\sigma}:[a,b]\to X\) such that \(\overline{\sigma}(a)=\overline{p},\overline{\sigma}(b)=\overline{q}\) and such that \(d(\sigma(t),\overline{\sigma}(t))<\varepsilon\) for all \(t\in [a,b]\).
\end{enumerate}

First we notice that $P(A)$  is true for \(A:=\varepsilon/L(\sigma)\). Indeed, we have 
\(d(\sigma(a),\sigma(b))=L(\sigma)(b-a)=L(\sigma|_{[a,b]})\leq\varepsilon\).
For all \(\overline{q}\in X\) with \(d(\sigma(b),\overline{q})<\varepsilon\) we then have by the triangle inequality \(d(\sigma(a),\overline{q})\leq d(\sigma(a),\sigma(b))+d(\sigma(b),\overline{q})<2\varepsilon\).
This shows that all \(\overline{p},\overline{q}\in X\) as in the claim are in the ball \(\overline{B}(\sigma(a),2\varepsilon)\), which by hypothesis is geodesically convex.
Therefore, there exists a local geodesic from \(\overline{p}\) to \(\overline{q}\) which satisfies \(d(\sigma(t),\overline{\sigma}(t))<\varepsilon\) by convexity of \(d\) on \(B(\sigma(a),2\varepsilon)\), proving \(P(\varepsilon/L(\sigma))\).

\medskip
Now we know that $P(A)$ is true for at least one positive value of \(A\). If we can now show that \(P(A)\) implies \(P(3A/2)\), then the set of \(A\)'s for which the statement is true is unbounded, i.e. it is true for all \(A>0\). Indeed, assume $P(A)$ is true for some \(A>0\).
Let \(a,b\in [0,1]\) be such that \(0<b-a<3A/2\).
Next, split the interval \([a,b]\) into three smaller intervals with endpoints \(a<a_1<b_1<b\) of the same length.
Notice that \(b_1-a<A\) and \(b-a_1<A\).
Let \(\overline{p},\overline{q}\in X\) be as in the statement \(P(3A/2)\).
Define \(p_0=\sigma(a_1),q_0=\sigma(b_1)\).
Now we assumed that \(P(A)\) is true, so once we define points \(p_{n-1}\in X\) and \(q_{n-1}\in X\) we know there are local geodesics \(\sigma_n:[a,b_1]\to X\) and local geodesics \(\sigma'_n:[a_1,b]\to X\) joining \(\overline{p}\) to \(q_{n-1}\) and joining \(p_{n-1}\) to \(\overline{q}\), respectively, such that \(d(\sigma(t),\sigma_n(t))<\varepsilon\) for \(t\in [a,b_1]\) and \(d(\sigma(t),\sigma'_n(t))<\varepsilon\) for \(t\in [a_1,b]\).
We then define \(p_n=\sigma_n(a_1)\) and \(q_n=\sigma'_n(b_1)\).
Because we chose \(\varepsilon>0\) to be small enough for the metric to be convex we have 
\begin{align*}
d(p_0,p_1)&=d(\sigma(a_1),\sigma_1(a_1))\\
&\leq \frac{1}{2}d(\sigma(a),\sigma_1(a))+\frac{1}{2}d(\sigma(b_1),\sigma_1(b_1))\\
&<\frac{1}{2}\varepsilon
\end{align*}
and
\begin{align*}
d(p_n,p_{n+1})&=d(\sigma_n(a_1),\sigma_{n+1}(a_1))\\
&\leq\frac{1}{2}d(\sigma_n(a),\sigma_{n+1}(a))+\frac{1}{2}d(\sigma_n(b_1),\sigma_{n+1}(b_1))\\
&=\frac{1}{2}d(q_n,q_{n-1}).
\end{align*}
Similarly we get \(d(q_0,q_1)<\varepsilon/2\) and \(d(q_n,q_{n+1})\leq \frac{1}{2}d(p_n,p_{n-1})\).
Recursive substitution then yields
\begin{align*}
d(p_n,p_{n+1})&<\frac{\varepsilon}{2^{n+1}}\\
d(q_n,q_{n+1})&<\frac{\varepsilon}{2^{n+1}},
\end{align*}
showing that \((p_n)_{n\geq 0}\) and \((q_n)_{n\geq 0}\) are Cauchy sequences and therefore converge to unique points \(p_\infty\) and \(q_\infty\), respectively.

As \(t\mapsto d(\sigma_n(t),\sigma_{n+1}(t))\) is convex, it is bounded from above by \(d(\sigma_n(b_1),\sigma_{n+1}(b_1))<\frac{\varepsilon}{2^{n+1}}\).
This means that \((\sigma_n(t))_{n\geq 0}\) also is a Cauchy sequence in the complete ball \(\overline{B}(\sigma(t),\varepsilon)\) for all \(t\in [a,b_1]\).
In the same way \((\sigma'_n(t))_{n\geq 0}\) is a Cauchy sequence in the complete ball \(\overline{B}(\sigma'(t),\varepsilon)\) for all \(t\in [a_1,b]\).
The local geodesics \(\sigma_n\) and \(\sigma'_n\) therefore converge uniformly to local geodesics \(\sigma_\infty:[a,b_1]\to X\) and \(\sigma'_\infty:[a_1,b]\).
These local geodesics restricted to the interval \([a_1,b_1]\) are two local geodesics between \(p_\infty,q_\infty\).
As \(b_1-a_1<A/2<A\), any local geodesic between these two points must be unique as shown before, so \(\sigma_\infty,\sigma_\infty'\) must coincide on \([a_1,b_1]\).
Now these make up a local geodesic from \(\overline{p}\) to \(\overline{q}\) given by
\[\overline{\sigma}=\begin{cases}
\sigma_\infty(t)   &\text{if }t\in [a,b_1]\\
\sigma'_\infty(t)   &\text{if }t\in[a_1,b]\end{cases}.\]
This finishes the proof.
\end{proof}

\subsection{The exponential map}
This subsection will be devoted to constructing a universal cover (and its corresponding covering map), from which we can deduce parts (2) and (3) of the Cartan--Hadamard Theorem.
\begin{definition}
Let \((X,d)\) be a metric space, and let \(x_0\in X\) be a point.
We then define the topological space \(\twiddle{X}_{x_0}\) as the set of linearly reparametrized local geodesics \(\sigma:[0,1]\to X\) that start at \(\sigma(0)=x_0\) and end at some point \(\sigma(1)\in X\), including the constant map \(x_0:[0,1]\to X\).
We then define a metric $d_s$ on \(\twiddle{X}_{x_0}\) by
\[d_s(\sigma,\sigma')=\sup_{t\in[0,1]}d(\sigma(t),\sigma'(t)).\]
and define \(\exp:\twiddle{X}_{x_0}\to X\) as the map \(c\mapsto \exp(c)=c(1)\).
\end{definition}
The map \(d_s\) is obviously symmetric and positively definite.
By the triangle inequality we also have for \(\sigma_1,\sigma_2,\sigma_3\in\twiddle{X}_{x_0}\):
\begin{align*}
d_s(\sigma_1,\sigma_2)+d_s(\sigma_2,\sigma_3)
&=\sup_{t\in[0,1]}d(\sigma_1(t),\sigma_2(t))+\sup_{t\in[0,1]}d(\sigma_2(t),\sigma_3(t))\\
&\geq\sup_{t\in[0,1]}d(\sigma_1(t),\sigma_3(t))\\
&=d_s(\sigma_1,\sigma_3)
\end{align*}
showing that \((\twiddle{X}_{x_0},d_s)\) is a metric space.
\begin{definition}
   We denote by $\twiddle{d}_s$ the induced length metric on $\twiddle{X}_{x_0}$ associated with $d_s$.
\end{definition}

\begin{remark}
 Notice that if \((X,d)\) is a path-connected metric space with locally complete and locally convex metric \(d\), then by Lemma \ref{lem:induced_convex}, for any two points $x,y \in X$ one can construct a local geodesic in $X$ between $x$ and $y$. Indeed, take a continuous curve $c$ between $x$ and $y$, that exists by the assumptions on $(X,d)$. Then, since $(X,d)$ is locally complete and locally convex, using Lemma \ref{lem:induced_convex} and the continuity of the curve $c$ one can construct step by step a local geodesic between $x$ and $c(t)$, for every $t \in [0,1]$.
\end{remark}

\begin{proposition}\label{prop:twiddle simply connected}
Let \((X,d)\) be a path-connected metric space with locally complete and locally convex metric \(d\). Let $x_0 \in X$.
\parenum
\begin{enumerate}
    \item \(\twiddle{X}_{x_0}\) is contractible, thus simply connected;
    \item \(\exp:\left(\twiddle{X}_{x_0},d_s\right)\to (X,d)\) is a local isometry;
    \item There is a unique local geodesic joining \(\twiddle{x}_0\) to any  point in \(\left(\twiddle{X}_{x_0},d_s\right)\).
\end{enumerate}
\end{proposition}
\begin{proof}
\parenum
\begin{enumerate}
    \item Let \(c\in\twiddle{X}_{x_0}\) and let \(s\in[0,1]\).
We then define the path \(r_s(c):[0,1]\to X\) as the map \(t\mapsto c(st)\).
The map \(\twiddle{X}_{x_0}\times [0,1]\to \twiddle{X}_{x_0}\) given by \((c,s)\mapsto r_s(c)\) is then a homotopy from the identity map in \(\twiddle{X}_{x_0}\) to the constant map \(\twiddle{x}_0\), showing that \(\twiddle{X}_{x_0}\) is contractible.
As all contractible spaces are simply connected we are done.
    \item Let \(\sigma\in\twiddle{X}_{x_0}\) be a local geodesic in \(X\) between \(x_0\in X\) and \(\sigma(1)\in X\).
By local completeness and local convexity of the metric there is an \(\varepsilon_\sigma>0\) such that the induced metric on \(B(\sigma(t), 2\varepsilon_\sigma)\) is complete and convex,  for every $t \in [0,1]$.
Now take another local geodesic \(\sigma'\) from \(x_0\) to \(\sigma'(1) \in B(\sigma(1),\varepsilon_\sigma)\).
We want to show that $$d_s(\sigma,\sigma')=d(\exp(\sigma),\exp(\sigma'))=d(\sigma(1),\sigma'(1)),$$ and thus  \( \sigma' \in B(\sigma,\varepsilon_\sigma) \subset \twiddle{X}_{x_0}\).
Indeed, by convexity from Lemma \ref{lem:induced_convex} we have \(d(\sigma(t),\sigma'(t))\leq t d(\sigma(1),\sigma'(1)) \leq d(\sigma(1),\sigma'(1))\), for every $t \in [0,1]$.
This gives us
\[d_s(\sigma,\sigma')=\sup_{t\in [0,1]} d(\sigma(t),\sigma'(t))=d(\sigma(1),\sigma'(1))=d(\exp(\sigma),\exp(\sigma')),\]
showing that $\exp$ restricted to the ball \(B(\sigma,\varepsilon_\sigma)\) is a distance preserving function onto \(B(\sigma(1),\varepsilon_\sigma)\). 
This map is clearly surjective, and injective by Lemma \ref{lem:induced_convex}, so it is a bijection, hence an isometry.
This holds for any \(\sigma\in \twiddle{X}_{x_0}\), so every point in \(\twiddle{X}_{x_0}\) has an open neighborhood which is isometric to an open neighborhood in \(X\), whereby $\exp$ is a local isometry.
    \item The image of a continuous curve \(\twiddle{c}\) in \(\twiddle{X}_{x_0}\) between two local geodesics \(\sigma:[0,1]\to X\) and \(\sigma':[0,1]\to X\) with \(\sigma(0)=\sigma'(0)=x_0\) under \(\exp\) is a continuous curve \(c\) from \(\sigma(1)\in X\) to \(\sigma'(1)\in X\).
    
As \(\exp\) is a local isometry, \(\twiddle{c}\) is a local geodesic if and only if \(c\) is. In particular, one can easily verify that we have a bijection between the set of local geodesics issuing from \(\twiddle{x}_0\) to the set of local geodesics issuing from \(x_0\) given by \(\twiddle{c}\mapsto\exp\circ\twiddle{c}\).
By hypothesis \(d\) is locally complete and locally convex, so we apply Lemma \ref{lem:induced_convex} to show that for each point \(c\in\twiddle{X}_{x_0}\), the image under \(\exp\) of the continuous path \(\twiddle{c}:[0,1]\to \twiddle{X}_{x_0}\) given by \(\twiddle{c}(s)=r_s(c)\) is the unique local geodesic joining \(\twiddle{x}_0\) to \(c\) in \(\twiddle{X}_{x_0}\),  still $c$ might not be the unique local geodesic in $X$ joining \(x_0\) to \(c(1)\) in \(X\).
\end{enumerate}
\end{proof} 

\begin{lemma}\label{lem:d_twiddle_complete}
    Let \((X,d)\) be a metric space with locally convex and complete metric \(d\). 
    Then for any \(x_0\in X\), the metric \(d_s\) on \(\twiddle{X}_{x_0}\) is complete.
\end{lemma}
\begin{proof}
    Let \((c_n)\) be a Cauchy sequence in \(\twiddle{X}_{x_0}\).
By definition for  any there is \(\varepsilon>0\) an \(N_\varepsilon\in\N\) such that
\[d_s(c_n,c_m)=\sup\{d(c_n(t),c_m(t)):t\in [0,1]\} < \varepsilon \quad\text{ for all }n,m\geq N_\varepsilon.\]
For a given \(\varepsilon>0\) we then have \(d(c_n(t),c_m(t))<\varepsilon\) for all \(t\in[0,1]\) and for all \(n,m\geq N_\varepsilon\).
In particular, the sequence \((c_n(t))_n\) is Cauchy in \(X\) for all \(t\in [0,1]\).
As \(d\) is assumed complete, \((c_n(t))_n\) converges to some point \(c(t)\in X\).
To show that \(\twiddle{X}_{x_0}\) is complete we just need to show that \(c:[0,1]\to X\), where \(c(0)=c_n(0)=x_0\), is a local geodesic.
As all the \(c_n\) are continuous, so is \(c\).
It remains to show that \(c\) is a local geodesic.
Let \(t\in [0,1]\).
Because \(X\) is locally convex, we can by definition find an open ball \(B(c(t),2\varepsilon)\subseteq X\) where \(d|_{B(c(t),2\varepsilon)}\) is a convex metric.
As \(c\) is continuous, we can find some sufficiently small interval \([t-\delta,t+\delta]\) such that
\begin{itemize}
    \item \(c(v)\in B(c(t),\varepsilon)\) for any \(v\in [t-\delta,t+\delta]\),
    \item we can find an \(N>0\) such that \(c_n(v)\in B(c(v),\varepsilon)\) for all \(n\geq N_\varepsilon\) and all \(v\in [t-\delta,t+\delta]\).
\end{itemize}
By the triangle inequality we then have \(c_n(v)\in B(c(t),2\varepsilon)\) for all \(v\in [t-\delta,t+\delta]\).
Now the restriction \(c_n:[t-\delta,t+\delta]\to X\), which is a local geodesic, is entirely contained in \(B(c(t),2\varepsilon)\).
Now \((B(c(t),2\varepsilon),d)\) is convex and geodesic, so there is a geodesic \(\gamma_n\) from \(c_n(t-\delta)\) to \(c_n(t+\delta)\) entirely contained in \(B(c(t),2\varepsilon)\).
By convexity \(\gamma_n\) is unique.
Now the function \(v\mapsto d(c_n(v),\gamma_n(v))\), where \(v\in [t-\delta,t+\delta]\), is locally convex, and therefore convex.
This in turn means that the restriction \(c_n:[t-\delta,t+\delta]\to X\) is a geodesic,  $c_n=\gamma_n$ on $[t-\delta,t+\delta]$.
Between the points \(c(t-\delta)\) and \(c(t+\delta)\) there is a unique geodesic \(\gamma:[t-\delta,t+\delta]\to X\) which is entirely contained in \(B(c(t),2\varepsilon)\) by choice of \(\varepsilon\).
We claim that this geodesic \textit{is} \(c\) restricted to the interval, i.e. \(\gamma(v)=c(v)\) for \(v\in [t-\delta,t+\delta]\), which in turn will show that \(c\) is a local geodesic.
Indeed \(d(c(v),\gamma(v))\leq d(c_n(v),c(v))+d(c_n(v),\gamma(v))\) for all \(n\geq N_\varepsilon\) by the triangle inequality.
As \(c_n,\gamma:[t-\delta,t+\delta]\to X\) are geodesics we furthermore get by convexity
\[d(c(v),\gamma(v))\leq d(c_n(v),c(v))+vd(c_n(t-\delta),\gamma(t-\delta))+(1-v)d(c_n(t+\delta),\gamma(t+\delta)).\]
Letting \(n\) go to infinity the right hand side goes to 0, so \(d(c(v),\gamma(v))=0\) and therefore \(c(v)=\gamma(v)\) for all \(v\in [t-\delta,t+\delta]\), finishing the proof.
\end{proof}

\begin{corollary}\label{cor:unique_geo}
    Let \((X,d)\) be a path-connected length space with complete and locally convex metric \(d\). Let \(x_0\in X\) be a point.
\parenum
\begin{enumerate}
    \item \(\exp:(\twiddle{X}_{x_0},\twiddle{d}_s )\to (X,d)\) is a universal covering map;
    \item There is a unique local geodesic joining each pair of points in \(\left(\twiddle{X}_{x_0},\twiddle{d}_s\right)\).
\end{enumerate}
\end{corollary}
\begin{proof}
\parenum
\begin{enumerate}
    \item First we note the following:
\begin{itemize}
    \item  The metrics $d_s$ and $\twiddle{d}_s$ on $\twiddle{X}_{x_0}$ locally agree on $\twiddle{X}_{x_0}$.
    \item \(X\) is assumed connected.
    \item $\exp$ is a local isometry.
    In particular, the map $\exp$ is a local homeomorphism where the length of every path in \(\twiddle{X}_{x_0}\) is no larger than the length of its image under $\exp$.
    \item \(X\) is locally uniquely geodesic by Lemma \ref{lem:induced_convex}, which also tells us that geodesics vary continuously with their endpoints locally (by convexity). 
    \item By Lemma \ref{lem:d_twiddle_complete} we know that the metric \(d_s\) on \(\twiddle{X}_{x_0}\) is complete; thus also $\twiddle{d}_s$.
    \item  By definition and construction  $(\twiddle{X}_{x_0},\twiddle{d}_s)$ is a length space. 
\end{itemize}
We can thus apply Proposition 3.28 from Chapter I.3 in \cite{BridHaef}, showing that $\exp$ is a covering map.
To show that \(\twiddle{X}_{x_0}\) is indeed a universal cover we just need to show that it is simply connected.
This is just Proposition \ref{prop:twiddle simply connected}.
    \item First we show that any continuous curve \(\sigma\) in \(X\) between two given points is homotopic to a unique local geodesic between said points.
As \(x_0\) is chosen arbitrarily, we can chose $x_0$ to be one of the two endpoints of \(\sigma\).
Now \(\twiddle{X}_{x_0}\) is a universal covering, so there is a bijection between the set of continuous paths homotopic to \(\sigma\) (with the same endpoints) and the set of continuous paths in \(\twiddle{X}_{x_0}\) from \(\twiddle{x}_0\) and having the same endpoint as the lift in $\twiddle{X}_{x_0}$ of the curve \(\sigma\) (apply again Lemma \ref{lem:induced_convex}).
By Proposition \ref{prop:twiddle simply connected} there is a unique local geodesic in \(\twiddle{X}_{x_0}\) with those endpoints (the uniqueness come again from the construction in Lemma \ref{lem:induced_convex}). So by Proposition \ref{prop:twiddle simply connected}(3),  the set of continuous paths in \(\twiddle{X}_{x_0}\) from \(\twiddle{x}_0\) and having the same endpoint as the lift in $\twiddle{X}_{x_0}$ of the curve $\sigma$ contains a unique local geodesic.

Next, let \(p,q\in\twiddle{X}_{x_0}\) be local geodesics.
We know that \(\twiddle{X}_{x_0}\) is simply connected, so it has only one homotopy class of curves with these endpoints.
This in turn means that the projection on \(X\) by the map $\exp$ will send that unique class of homopotic curves bijectively onto one single homotopy class of curves in $X$ between the corresponding endpoints.
We have showed above that a homotopy class of curves in $X$ contains a unique local geodesic \(\overline{\sigma}:[0,1]\to X\), so the corresponding homotopy class in \(\twiddle{X}_{x_0}\) must also contain a unique local geodesic, the lifting of \(\overline{\sigma}\).
This is then the unique geodesic between \(p\) and \(q\).
\end{enumerate}
\end{proof}

 We have now shown that \(\twiddle{X}_{x_0}\) is the universal cover of \(X\) with covering map \(\exp:\twiddle{X}_{x_0}\to X\), and that $\twiddle{X}_{x_0}$ is uniquely locally geodesic. Since \((\twiddle{X}_{x_0}, \twiddle{d}_s)\) is a length space as well, by Proposition \ref{prop:p_loc_iso} we have that \((\twiddle{X}_{x_0}, \twiddle{d}_s)\) is the same space as $(\twiddle{X}, \twiddle{d})$ introduced in Definition \ref{def::ind_metric_univ}, and the corresponding covering map $p$ is the $\exp$ map. This is exactly part (3) of the Cartan--Hadamard Theorem (Theorem \ref{thm:CHT}).

It remains to show parts (2) and (3) of the Cartan--Hadamard Theorem, i.e. that for a metric space \((X,d)\), if the metric \(d\) is locally convex, then the metric \(\twiddle{d}_s=\twiddle{d}\) on the universal covering \((\twiddle{X}_{x_0},\twiddle{d}_s) \cong (\twiddle{X},\twiddle{d})\) is globally convex and uniquely geodesic.
\begin{proof}[Proof of parts (2) and (3) of Theorem \ref{thm:CHT}]
We have shown that \(\twiddle{X}\) is a complete uniquely local geodesic space in which  local geodesics vary continuously with their endpoints.
Let \(p,q\in \twiddle{X}\) be points and denote the unique local geodesic joining \(p\) and \(q\) by \(\sigma_{p,q}:[0,1]\to \twiddle{X}\).  Then one can apply the easy argument from the proof of \cite[Lemma 4.8(1)]{BridHaef} to show that in fact $\sigma_{p,q}$ is a geodesic, thus also unique. So part (3) of the Cartan--Hadamard Theorem follows.

Now we show global convexity of $\twiddle{d}$.
If we can show \(\twiddle{d}(\sigma_{p,q_0}(1/2),\sigma_{p,q_1}(1/2))\leq (1/2)\twiddle{d}(q_0,q_1)\) for each pair \(\sigma_{p,q_0},\sigma_{p,q_1}:[0,1]\to \twiddle{X}\), we are done.
Then from there, we can prove by induction that \(\twiddle{d}(\sigma_{p,q_0}(t),\sigma_{p,q_1}(t))\) for all dyadic rationals in \([0,1]\). 
As geodesics vary continuously with their endpoints, this extends to all \(t\in [0,1]\).

To this end, suppose that \(q_1,q_2 \in \twiddle{X} \) are connected by the geodesic \(s\mapsto q_s\).
By Lemma \ref{lem:induced_convex} part (1) we then know that \(\twiddle{d}(\sigma_{p,q_s}(1/2),\sigma_{p,q_{s'}}(1/2))\leq (1/2)\twiddle{d}(q_s,q_{s'})\) whenever \(s,s'\) are sufficiently close.
Now we partition \([0,1]\) to \(0=s_0<s_1<\ldots<s_n=1\) finely enough that the above inequality holds for all \(\{s,s'\}=\{s_i,s_{i+1}\}\).
Now we just add all these inequalities to get the desired \(\twiddle{d}(\sigma_{p,q_0}(1/2),\sigma_{p,q_1}))\leq (1/2)\twiddle{d}(q_0,q_1)\).
\end{proof}

\subsection{Alexandrov's Patchwork}
For the final part of the proof of Theorem \ref{thm:CHT} we will make use of the characterization of CAT(0) spaces in terms of angles.
We will first show that ``gluing'' two geodesic triangles whose angles are no greater than their comparison angles yields a larger triangle with the same property:
\begin{lemma}[Gluing Lemma]\label{lem:gluing}
    Let \(X\) be a geodesic metric space, and let \(\Delta\) be a geodesic triangle with distinct vertices \(p,q_1\) and \(q_2\) in \(X\).
    For \(r\in[q_1,q_2]\), let \(\Delta_1,\Delta_2\) be the geodesic triangles with vertices \(p,q_1,r\) and \(p,r,q_2\), respectively.
    If the vertex angles of the \(\Delta_i\) are no greater than the corresponding vertex angles in their comparison triangles, then the same is true for \(\Delta\).
\end{lemma}
\begin{proof}
    Let \(\overline{\Delta}_1=\overline{\Delta}_1(\overline{p},\overline{q}_1,\overline{r})\) and \(\overline{\Delta}_2=\overline{\Delta}_2(\overline{p},\overline{r},\overline{q}_2)\) be comparison triangles in \(\E^2\) for \(\Delta_1\) and \(\Delta_2\) arranged such that they share the common side \([\overline{p},\overline{r}]\) with \(\overline{q}_1\) and \(\overline{q}_2\) on opposite sides.
    As the Alexandrov angle is a pseudometric, the sum of the angles of the \(\Delta_i\) at \(r\) is at least \(\angle_r(q_1,q_2)=\pi\). 
    By assumption the same is then true for the sum of the comparison angles.
    Alexandrov's Lemma (Lemma \ref{lem:alex}) now tells us that \(\overline{\angle}_{q_1}(p,q_2)\geq \overline{\angle}_{q_1}(p,r)\), which shows that the angle at \(q_1\) in \(\Delta\) is no greater than its comparison angle.
    A similar argument shows the same for \(q_2\).
    For the angle at \(p\), notice that Alexandrov's Lemma implies that the distances between the vertices corresponding to \(q_1\) and \(q_2\) in the comparison triangle of \(\Delta\) is greater than or equal to \(d(\overline{q_1},\overline{q_2})\), which furthermore implies
\begin{align*}
    \overline{\angle}_p(q_1,q_2)&\geq \overline{\angle}_p(q_1,r)+\overline{\angle}_p(r,q_2)\\
    &\geq \angle_p(q_1,r)+\angle_p(r,q_2)\tag{by assumption}\\
    &\geq \angle_p(q_1,q_2),\tag{by the triangle inequality}
\end{align*}
finishing the proof.
\end{proof}

In locally CAT(0) spaces one can imagine cutting any geodesic triangle into a collection of suitably small geodesic triangles and then repeatedly applying the Gluing Lemma (Lemma \ref{lem:gluing}) to show that the original triangle satisfies this angle/comparison angle inequality.
This is made precise in the following lemma:
\begin{lemma}[Alexandrov's Patchwork]
    Let \((X,d)\) be a geodesic metric space which is locally CAT(0) as in Definition \ref{def:loc_cat}.
    Let \(q_1,q_2\in X\) be points connected by the geodesic \(q:[0,1]\to X\) starting at \(q_1\) and ending at \(q_2\), and let \(p\in X\) be a point not on the image of \(q\).
    For \(s\in [0,1]\), let \(c_s:[0,1]\to X\) be a geodesic from \(p\) to \(q(s)\), and assume that \(s\mapsto c_s\) is continuous.
    Then the angles at \(p,q_0,q_1\) in the geodesic triangle \(\Delta\) with sides \(c_0([0,1]),c_1([0,1]),q([0,1])\) are no greater than the corresponding angles in the comparison triangle of \(\Delta\).
\end{lemma}

\begin{proof}
    By hypothesis the map \(c:[0,1]\times[0,1]\mapsto X\) mapping \((s,t)\) to \(c_s(t)\) is continuous, so we can create partitions \(0<s_0<s_1<\cdots<s_k=1\) and \(0<t_0<t_1<\cdots<t_k=1\) making \(c([s_{i-1},s_i]\times [t_{j-1},t_j])\) arbitrarily small for all \(0<i,j\leq k\).
    By hypothesis \(X\) is locally CAT(0), so we can then create these partitions such that all these \(c([s_{i-1},s_i]\times [t_{j-1},t_j])\) are contained in an open ball \(U_{ij}\) which is CAT(0).
    Construct a sequence of \(k\) geodesic triangles \(\Delta_1,\ldots,\Delta_k\) by letting \(\Delta_i\) be the geodesic triangle with sides \(c_{s_{i-1}}([0,1]),c_{s_i}([0,1])\) and \(q([s_{i-1},s_i])\).
    Next we cut each of these triangles into smaller triangles as follows:
    For each \(i\), let \(\Delta_i^1\) be the geodesic triangle in \(U_{1,1}\) with vertices \(p,c_{s_{i-j}}(t_1),c_{s_i}(t_1)\).
    For each \(i\), let \(\Delta_i^j\) be the geodesic triangle in \(U_{i,j}\) with vertices \(c_{s_{i-1}}(t_{j-1}),c_{s_i}(t_{j-1})\) and \(c_{s_{i-1}}(t_j)\).
    Finally, for each \(i\), let \(\twiddle{\Delta}_i^j\) be the geodesic triangle with vertices \(c_{s_{i-1}}(t_j),c_{s_i}(t_j)\) and \(c_{s_i}(t_{j-1})\).
    All of these are contained in a CAT(0) space, so their angles are no greater than their comparison angles.
    Repeatedly applying the gluing lemma to the zig-zag pattern created inbetween all the \(c_{s_i}\) shows that the \(\Delta_1,\ldots,\Delta_k\) satisfy the angle-comparison angle inequality.
    As these triangles patched together make up the whole geodesic triangle \(\Delta\), another repeated usage of the gluing lemma finishes the proof.
\end{proof}
\begin{figure}[H]
    \centering
    \includegraphics[width=\textwidth*2/3]{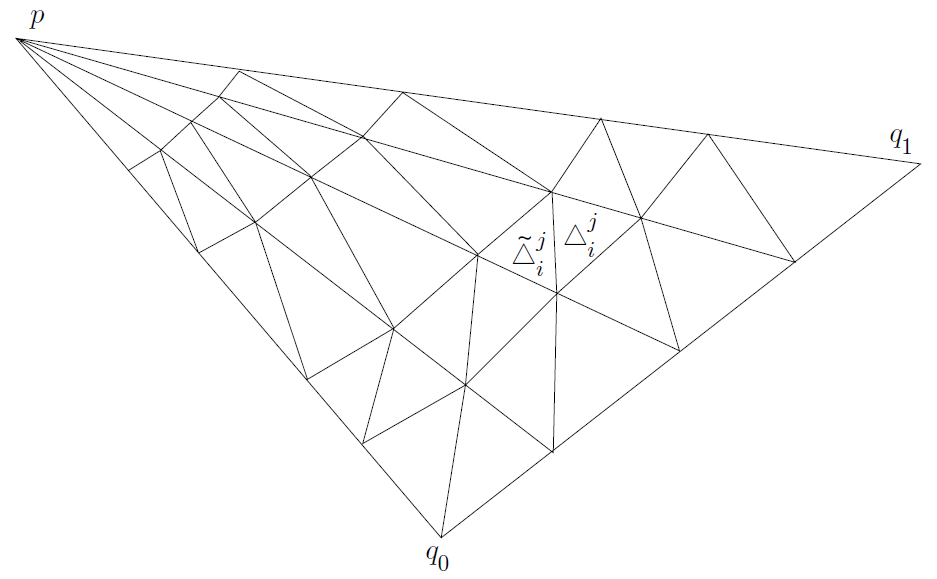}
    \caption{\cite{BridHaef} Alexandrov's patchwork: each of the small triangles are contained in CAT(0) spaces so they satisfy the angle/comparison angle inequality.}
\end{figure}

In the last section we showed that the universal covering map \(p\) is the local isometry $\exp$ when equipping \(\twiddle{X}_{x_0}\) with the metric induced by $\exp$.
This shows that if \((X,d)\) is locally CAT(0) then so is \((\twiddle{X}_{x_0},\twiddle{d}_s)\).
We also showed that \(\twiddle{X}_{x_0}\) is uniquely geodesic, so Alexandrov's patchwork shows us that all geodesic triangles in \(\twiddle{X}_{x_0}\) satisfy that their angles are no greater than the corresponding angles in the comparison triangle, proving that \(\twiddle{X}_{x_0}\) is a CAT(0) space. 
This finally proves the last part of the Cartan--Hadamard theorem.

\end{document}